\documentclass{amsart}

\usepackage{graphicx}
\usepackage{epstopdf}
\newtheorem{thm}{Theorem}[section]

\newtheorem{lem}[thm]{Lemma}
\newtheorem{prop}[thm]{Proposition}
\theoremstyle{definition}
\newtheorem{defn}[thm]{Definition}
\theoremstyle{remark}

\numberwithin{equation}{section}

\begin{document}
\title[Zeros of random Reinhardt polynomials]{Zeros of random Reinhardt polynomials}
\author{Arash Karami}
\begin{abstract}
For a strictly pseudoconvex  Reinhardt domain $\Omega$ with smooth boundary in $\mathbb{C}^{m+1}$ and a positive smooth measure $\mu$ on
the boundary of $\Omega$ , we consider the ensemble $\mathcal{P}_{N}$ of polynomials of degree $N$ with the Gaussian probability measure
$\gamma_{N}$ which is induced by $L^{2}(\partial\Omega,d\mu)$. Our aim is to compute the scaling limit distribution function and scaling limit
pair correlation function for zeros near a point $z\in\partial\Omega$. First, we apply the stationary phase method to the
Boutet de Monvel-Sj\"{o}strand Theorem to get the asymptotic for the scaling limit partial Szeg\"{o} kernel around $z\in\partial\Omega$.
Then by using the Kac-Rice formula, we compute the scaling limit distribution and pair correlation functions.
\end{abstract}
\maketitle

\section{introduction}
This paper is concerned with the scaling limit distribution and pair correlation between zeros of random polynomials on $\mathbb{C}^{m+1}$. A famous result from Hammersley \cite{Ha} which is the following work of Kac \cite{Kac1}, \cite{Kac2} says that the zeros of random complex Kac polynomials,
\begin{equation}
f(z)=\sum_{j\leq N}a_{j}z^{j},z\in\mathbb{C},
\label{intro1}
\end{equation}
tend to concentrate on the unit circle $S^{1}=\{z\in\mathbb{C}:|z|=1\}$ as the degree of the polynomials goes to infinity when the coefficients $a_{j}$ are independent complex Gaussian random variables of mean zero and variance one. Later Bloom and Shiffman in \cite{BS1} proved a multi-variable result that the common zeros of $m+1$ random complex polynomials in $ \mathbb{C}^{m+1}$,
\begin{equation}
f_{k}(z)=\sum_{|J|\leq k}c_{J}^{k}z_{0}^{j_{0}}\dots z_{m}^{j_{m}},
\label{intro2}
\end{equation}
tend to concentrate on the product of unit circles $|z_{j}|=1$. Shiffman in joint work with Zelditch in \cite{SZ} replaced $S^{1}$ with any closed
analytic curve $\partial\Omega$ in $\mathbb{C}$ that bounds a simply connected domain $\Omega$. In their work they used the Riemann mapping function $\Phi$ which maps the interior of $\Omega$ to the interior of the unit disk, mapping $z_{0}\in\partial\Omega$ to $1\in S^{1}$ and they let $\hat{D}^{N}_{\mu,\partial\Omega}:=D^{N}\circ\phi^{-1}|(\phi^{-1})'|^{2}$ be the expected zero density for the inner product with respect to the coordinate $w=\phi(z)$. So the new inner product on the space of holomorphic polynomials $P_{N}$ is
\begin{equation}
(f,g)_{\partial\Omega,\mu}=\int_{\partial\Omega}f\bar{g}d\mu(z),
\label{intro3}
\end{equation}
where $d\mu(z)$ is a positive smooth volume measure on $\partial\Omega$. Then with respect to this inner product, they proved that there is a scaling limit density function $D^{\infty}$ such that
\begin{equation}
\frac{1}{N^{2}}\hat{D}_{\partial\Omega,\mu}^{N}(1+\frac{u}{N})\rightarrow D^{\infty}(u),
\label{intro4}
\end{equation}
where $N\rightarrow\infty.$
They also showed that there exist universal functions $\hat{K}^{\infty}:\mathbb{C}^{2}\rightarrow\mathbb{R}$ independent of $\Omega,z_{0},\mu$ such that
\begin{equation}
\frac{1}{N^{4}}\hat{K}^{N}_{\partial\Omega,\mu}(1+\frac{u}{N},1+\frac{v}{N})\rightarrow K^{\infty}(u,v),
\label{intro5}
\end{equation}
as $N\rightarrow\infty$, where $\hat{K}^{N}_{\partial\Omega,\mu}=K^{N}_{\partial\Omega,\mu}\circ\Phi^{-1}$ is the pair correlation function written in terms of the complex coordinate $w=\phi(z)$. The first purpose of this paper is to compute the asymptotic expansion of the truncated Szeg\"{o} kernel on the boundary of the strictly pseudoconvex complete Reinhardt domain $\Omega$ in $\mathbb{C}^{m+1}$. Our second purpose is to generalize the scaling limit expected distribution result \cite{SZ} to the boundary of $\Omega$, and also to compute the pair correlation between zeros. First, we need to introduce our basic setting: We let $\Omega$ be a smooth strictly pseudoconvex complete Reinhardt domain (see Definition \eqref{s1}) in $\mathbb{C}^{m+1}$ and let $X=\partial\Omega$ and  $\mu$ be a smooth positive volume measure on $X$ that is invariant under the torus action,
\begin{equation}
(e^{i\theta_{0}},\dots,e^{i\theta_{m}})\cdot(z_{0},\dots,z_{m})=(e^{i\theta_{0}}z_{0},\dots,e^{i\theta_{m}}z_{m}),
\label{5.1}
\end{equation}
where $z=(z_{0},\dots,z_{m})\in X,\theta_{i}\in[0,2\pi]$. We give the space $\mathcal{P}_{N}$ of holomorphic polynomials of degree$\leq N$ on $\mathbb{C}^{m+1}$ the Gaussian probability measure $\gamma_{N}$ induced by the Hermitian inner product
\begin{equation}
(f,g)=\int_{X}f\bar{g}d\mu(x).
\label{intro6}
\end{equation}
The Gaussian measure  $\gamma_{N}$ induced from \eqref{intro6} can be described as follows: we write
\begin{equation}
f=\sum_{k=1}^{d(N)}a_{k}p_{k},
\label{7}
\end{equation}
where $\{p_{k}\}$ is the orthonormal basis of $\mathcal{P}_{N}$ with respect to inner product \eqref{intro6} and $d(N)=\dim\mathcal{P}_{N}$. Identifying $f\in \mathcal{P}_{N}$ with $a=(a_{k})\in\mathbb{C}^{d(N)}$, we have
\begin{equation}
d\gamma_{N}(a)=\frac{1}{\pi^{d(N)}}e^{-|a|^{2}}da.
\label{intro8}
\end{equation}
In other words, a random polynomial in the ensemble $(\mathcal{P}_{N},\gamma_{N})$ is a polynomial $f=\sum_{k=1}^{d(N)}a_{k}p_{k}$ such that the coefficients are independent complex Gaussian random variables with mean $0$ and variance $1$. Our first result, Theorem \eqref{intro1010}, gives an asymptotic for the scaling partial Szeg\"{o} kernel with respect to the inner product \eqref{intro6},
\begin{equation}
 S_{N}(z,w)=\sum_{k=1}^{d(N)}p_{k}(z)\bar{p}_{k}(w),
\label{intro9}
\end{equation}
 that gives the orthogonal projection onto the span of all homogeneous polynomials of degree$\leq N$.
\begin{thm}
If $z=(z_{0},\dots,z_{m})\in X\cap(\mathbb{C}^{*})^{m+1}$ and $u=(u_{0},\dots,u_{m})$, $v=(v_{0},\dots,v_{m})\in \mathbb{C}^{m+1}$ then
\begin{equation}
\lim_{N\rightarrow\infty}\frac{1}{N^{m+1}}S_{N}(z+\frac{u}{N},z+\frac{v}{N})=C_{\Omega, z,\mu,m}F_{m}(\beta(u)+\bar{\beta}(v)),
\end{equation}
\label{intro1010}
\end{thm}
where $C_{\Omega, z, \mu,m}$ is the constant that depends on $\Omega$, $z$, $\mu$, $m$ and
\begin{equation}
 F_{m}(t)=\int_{0}^{1}e^{ty}y^{m}dy \;\text{,}\; \beta(w)=\frac{d'\rho(z)\cdot w}{d'\rho(z)\cdot z}\:\text{,}\;
 d'\rho(z)=(\frac{\partial\rho}{\partial z_{1}},\dots,\frac{\partial\rho}{\partial z_{m}}).
\label{11}
\end{equation}
Our method to compute scaling asymptotic for the partial Szeg\"{o} kernel is similar to the method that Zelditch used in \cite{Ze}. In our proof we apply the stationary phase method to
\begin{equation}
 \Pi_{K}(x,y)=\int_{0}^{\infty}\int_{0}^{2\pi}e^{-i K\theta}e^{it\psi(e^{i\theta }x,y)}s(e^{i\theta}x,y,t)d\theta dt,
\label{intro12}
\end{equation}
where,
 $s(x,y,t)\sim \sum_{k=0}^{\infty} t^{m-k}s_{k}(x,y)$ and the phase $\psi\in C^{\infty}(\mathbb{C}^{m+1}\times\mathbb{C}^{m+1} )$ is determined by the following properties:

   1) $\psi(x,x)=\frac{\rho(x)}{i}$, where $\rho$ is the defining function of X,

   2) $\bar{\partial}_{x}\psi$ and $\partial_{y}\psi$ vanish to infinite order along the diagonal,

   3) $\psi(x,y)=-\bar{\psi}(y,x)$.\\

In \cite{BSZ1}, \cite{BSZ2} we see that the expected zero density and correlation functions can be represented by the formulas involving only the Szeg\"{o} kernel and its first and second derivatives. For each $f\in\mathcal{P}_{N}$ we associate the current of the integration
 $$[Z_{f}]\in D'^{1,1}(C^{m+1}),$$
 such that
 $$([Z_{f}],\psi)=\int_{Z_{f}}\psi\:\text{,}\:\:\:\psi\in D^{m,m}(C^{m+1}).$$
 In section \eqref{Scaling limit Distributions} we show that the scaling limit for  the expected zero density, which is defined by
\begin{equation}
D^{N}_{\mu,X}(z)\frac{\omega^{m+1}}{(m+1)!}=E_{\mu,X}^{N}([Z_{f}]\wedge\frac{\omega^{m}_{z}}{m!}),
\label{13}
\end{equation}
where $E_{\mu,X}^{N}$ is the expected zero current for the ensemble $(\mathcal{P}_{N},\gamma_{N})$ and $\omega_{z}=\frac{i}{2}\sum_{j=0}^{m}dz_{j}\wedge d\bar{z}_{j}$, can be given by the following Theorem.
\begin{thm}
Let $D_{\mu,X}^{N}$ be the expected zero density for the ensemble $(\mathcal{P}_{N},\gamma_{N})$. Then
 $$\lim_{N\rightarrow\infty}\frac{1}{N^{2}}D_{\mu,X}^{N}(z+\frac{u}{N})=D_{z,X}^{\infty}(u),$$
  where
 $$D_{z,X}^{\infty}(u)=\frac{\beta(P)}{\pi||P||^{2}}(\log F_{m})^{''}(\beta(u)+\bar{\beta}(u)),$$
and
$$P=(\frac{\partial\rho}{\partial\bar{z}_{0}},\dots,\frac{\partial\rho}{\partial\bar{z}_{m}}).$$
\label{intro14}
\end{thm}
Our main result, Theorem \eqref{intro17}, gives a formula for the scaling limit normalized pair correlation functions
\begin{equation}
\tilde{K}^{N}_{\mu,X}(z,w)=\frac{K^{N}_{\mu,X}(z,w)}{D^{N}_{\mu,X}(z)D^{N}_{\mu,X}(w)},
\label{intro15}
\end{equation}
where
\begin{equation}
K^{N}_{\mu,X}(z,w)\frac{\omega_{z}^{m+1}}{(m+1)!}\wedge\frac{\omega_{w}^{m+1}}{(m+1)!}=
E^{N}_{\mu,X}([Z_{f}(z)]\wedge[Z_{f}(w)]\wedge\frac{\omega_{z}^{m}}{(m)!}\wedge\frac{\omega_{w}^{m}}{(m)!}).
\label{intro16}
\end{equation}
If $z$, $w$ are fixed and different then $\tilde{K}^{N}_{\mu,X}(z,w)\rightarrow 1$ as $N\rightarrow\infty$, but in the Theorem \eqref{intro17} we show that we have nontrivial normalized pair correlations when the distance between points is $O(\frac{1}{N})$. To simplify our computations we define matrices
\begin{equation}
G_{m}(x)=\begin{pmatrix} F_{m}(x+\bar{x})&F_{m}(x)\\ F_{m}(\bar{x})&F_{m}(0)\end{pmatrix},
\end{equation}
\begin{equation}
Q_{m}(x)=G_{m+2}(x)-G_{m+1}(x)G_{m}(x)^{-1}G_{m+1}(x).
\end{equation}
\begin{thm}
Let $\tilde{K}_{\mu,X}^{N}(z,w)$ be the normalized pair correlation function for the probability space $(P_{N},\gamma_{N})$ and choose $u\in\mathbb{C}^{m+1}$ such that $u\notin T_{z}^{h}X$. Then,
$$\lim_{N\rightarrow\infty}\frac{1}{N^{4}}K_{\mu,X}^{N}(z+\frac{u}{N},z)=K^{\infty}_{z,X}(u),$$
$$\lim_{N\rightarrow\infty}\tilde{K}^{N}_{\mu,X}(z+\frac{u}{N},z)=\tilde{K}^{\infty}_{z,X}(u),$$
such that
$$K^{\infty}_{z,X}(u)=\frac{1}{\pi^{2}||P||^{4}}\frac{perm(Q_{m}(\beta(u)))}{det(G_{m}(\beta(u)))}(\beta(P))^{4},$$
$$\tilde{K}^{\infty}_{z,X}(u)=\frac{1}{(\log F_{m})^{''}(\beta(u)+\bar{\beta}(u))(\log F_{m})^{''}(0)}\frac{perm(Q_{m}(\beta(u)))}{det(G_{m}(\beta(u)))},$$
where $K^{N}_{\mu,X}(z,w)$, $\tilde{K}^{N}_{\mu,X}(z,w)$ are defined in \eqref{intro16}, \eqref{intro15}.
\label{intro17}
\end{thm}
  For fixed $z\in X\cap(\mathbb{C}^{*})^{m+1}$, $\beta$ is a $\mathbb{C}$-linear function on $\mathbb{C}^{m+1}$ that is independent of the defining function $\rho$. We see that
\begin{equation}
 \beta(u)=\frac{d'\rho(z)\cdot u}{d'\rho(z)\cdot z}=\frac{\sum_{i=0}^{m}(\frac{\partial\rho(z)}{\partial r_{i}}r_{i})\frac{u_{i}}{z_{i}}}{\sum_{i=0}^{m}\frac{\partial\rho(z)}{\partial r_{i}}r_{i}}.
\label{17.2}
\end{equation}
 So the function $\beta(u)$ can be interpreted as the weighted average of the $\frac{u_{i}}{z_{i}}$s with respect to the weights $\frac{\partial\rho(z)}{\partial r_{i}}r_{i}$. The argument of the $\frac{u_{i}}{z_{i}}$ measures the angle between the $i$'s component of the vector $u$ and the radial vector $z$. Therefore the imaginary part of the $\beta(u)$ is equal to the weighted average of the $\sin(\arg(\frac{u_{i}}{z_{i}}))$. In the radial direction, $u=z$, and the normal direction, $u=d''\rho(z)$, the angle $\arg(\frac{u_{i}}{z_{i}})$ is zero for each component. Hence we expect no oscillation for the graph of the normalized pair correlation functions in those two directions. However for the directions with nonzero weighted average of the $\sin(\arg(\frac{u_{i}}{z_{i}}))$, we expect oscillation in the graph, higher weighted average results in the higher frequency.
 It is interesting to see the behavior of the normalized pair correlation function in the normal direction. For example if we look at the sphere $S^{3}$ in the $\mathbb{C}^{2}$ and choose $z=(1,0)\in S^{3}\subset\mathbb{C}^{2}$ then the normal vector at $(1,0)$ to $S^{3}$ would be $u^{\bot}=(1,0)$. If we move along this vector from the origin to infinity then, in the normal direction, we obtain the scaling limit
\begin{equation}
k^{\bot}(\lambda):=\tilde{K}^{\infty}_{(1,0),S^{3}}(\lambda u^{\bot})=\lim_{N\rightarrow\infty}\tilde{K}^{N}_{\mu,S^{3}}((1,0)+\lambda\frac{u^{\bot}}{N},(1,0)).
\label{18}
\end{equation}
The graph of $k^{\bot}(\lambda)$ in Figure 1 converges to 1 when $\lambda$ goes to infinity. It is not oscillatory and we have a zero repulsion when $\lambda\rightarrow 0$.
\begin{figure}
\centering
\includegraphics{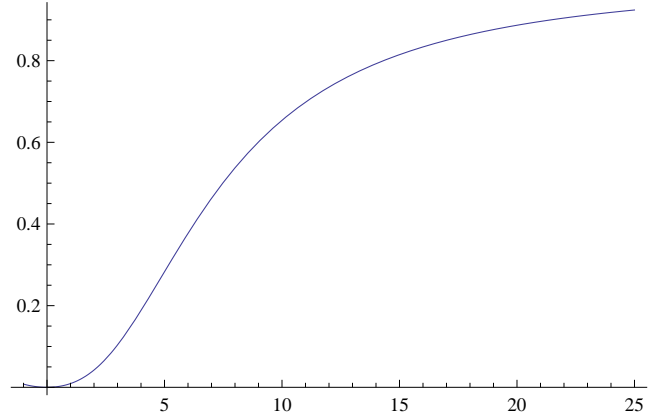}
\caption{The normalized pair correlation function $k^{\bot}(\lambda)$ in the normal direction $u^{\perp}$ for the sphere in $\mathbb{C}^{2}$ }
\end{figure}
It is interesting to measure the probability of finding a pair of zeros in the small disks around two points on $X$ in terms of scaled angular distance $\theta$ between them. In this example to consider the scaling limit for the pair correlation function in the $\frac{\partial}{\partial\theta}$ direction, we move along the curve $\gamma(\theta)=e^{i\theta}(1,0)$. The vector $u^{\theta}=(i,0)$ is the tangent vector to this curve at $\gamma(0)=(1,0)$. We observe that
\begin{equation}
K^{\infty}_{(1,0),S^{3}}(u^{\theta})=\lim_{N\rightarrow\infty}\frac{1}{N^{4}}K_{\mu,S^{3}}^{N}((1,0)+\frac{u^{\theta}}{N},(1,0)),
\label{19}
\end{equation}
\begin{equation}
k^{\theta}(\lambda):=\tilde{K}^{\infty}_{(1,0),S^{3}}(\lambda u^{\theta})=\lim_{N\rightarrow\infty}\tilde{K}^{N}_{\mu,S^{3}}((1,0)+\lambda\frac{u^{\theta}}{N},(1,0)).
\label{20}
\end{equation}
This means that the scaling limit pair correlation function grows as fast as $N^{4}$ along the curve $\gamma(\theta)$. We can see in the graph of $k^{\theta}$ in Figure 2, the zeros repel when $\lambda\rightarrow 0$ and their correlations are oscillatory.
\begin{figure}
\centering
\includegraphics{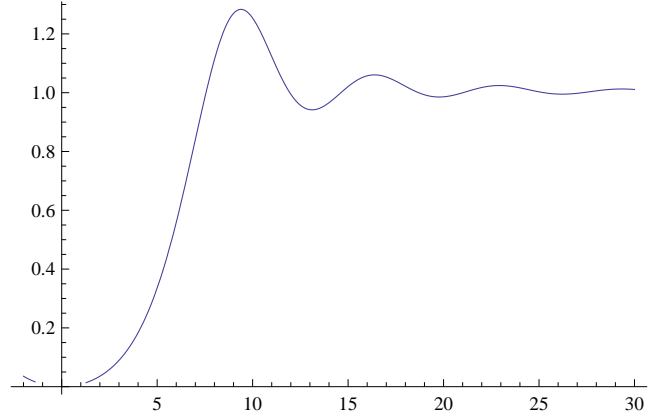}
\caption{The normalized pair correlation function $k^{\theta}(\lambda)$ in the $\frac{\partial}{\partial\theta}$ tangent direction $u^{\theta}$ for the sphere in  $\mathbb{C}^{2}$}
\end{figure}
Now if we move along $h(t)=(\cos(t),i\sin(t))\subset S^{3}$, then
\begin{equation}
\lim_{N\rightarrow\infty}\frac{1}{N^{5}}K_{\mu,S^{3}}^{N}((1,0)+\frac{u^{h}}{N},(1,0))\rightarrow K^{\infty}_{(1,0),S^{3}}(u^{h}),
\label{21}
\end{equation}
 where $u^{h}=h^{'}(0)=(0,i)$, $u^{h}=(0,i)\in T_{z}^{h}S^{3}$. The behavior of the scaling pair correlation function between zeros is totally
different when we move in the $u^{h}$ direction compare to $u^{\perp}$, and $u^{\theta}$. In this example we observe that if we move along the $u^{h}$ direction that belongs to $T_{z}^{h}S^{3}$ then $K_{\mu,S^{3}}^{N}((1,0)+\frac{u^{h}}{N},(1,0))$ is asymptotic to $N^{5}$, but in other directions, $K_{\mu,S^{3}}^{N}((1,0)+\frac{u}{N},(1,0))$ is asymptotic to $N^{4}$. Our result shows that $K_{\mu,X}^{N}(z+\frac{u}{N},z)$ is asymptotic to $N^{4}$ when $u\notin T_{z}^{h}X$.
\subsection{Acknowledgements}
 I am thankful to B.Shiffman and H.Hezari for their helpful advice and comments.
\section{Asymptotics of orthogonal polynomials}
\label{Asymptotics of orthogonal polynomials}
Throughout this paper, we restrict ourselves to a smooth boundary complete Reinhardt strictly pseudoconvex domain in $\mathbb{C}^{m+1}$.
 This is by far one of the most interesting cases to study, and it includes many interesting examples.
 We recall the elementary definitions:
\begin{defn}
A domain $\Omega$ is strictly pseudoconvex if its Levi form is strictly positive definite at every boundary point. The Levi form of $$\Omega=\{z\in\mathbb{C}^{m+1}:\rho(z)<0\},$$
with $\rho$ is a real valued $C^{\infty}$ function on $\mathbb{C}^{m+1}$ , $d'\rho\neq 0$ on $\partial\Omega$ defined as the restriction of the quadratic form
$$(v_{0},\dots,v_{m})\rightarrow\sum_{j,k}\frac{\partial^{2}\rho}{\partial z_{j}\partial\bar{z}_{k}}(z)v_{j}\bar v_{k},$$
to the subspace $\{(v_{0},\dots,v_{m})\in\mathbb{C}^{m+1}:\sum\frac{\partial\rho}{\partial z_{j}}(z)z_{j}=0\}$. It is defined independently of $\rho$  up to constants \cite{BFG}.
\end{defn}
\begin{defn}
A domain $\Omega\subset\mathbb{C}^{m+1}$ is complete Reinhardt if  $z=(z_{0},\dots,z_{m})\in\Omega$ implies $(\mu_{0}z_{0},\dots,\mu_{m}z_{m})\in\Omega$ for all $\mu_{j}\in\mathbb{C}$ with $|\mu_{j}|\leq 1,j=0,\dots,m$ \cite{SK}.
\label{s1}
\end{defn}
Throughout this article we assume that $d\mu$ is a smooth volume measure on $\partial\Omega$ that is invariant under the torus action. In the next section I will review some background materials from \cite{SK}

\subsection{Szeg\"{o} kernel and orthogonal polynomials}
Let $A(\Omega)$ be the space of holomorphic functions in $\Omega$ that extend continuously on the boundary. We define $H^{2}(\partial\Omega)$ to be the closure of the restriction of the functions in $A(\Omega)$ in $L^{2}(\partial\Omega,d\mu)$ \cite{SK}. So $H^{2}(\partial\Omega)$ is a proper closed subspace of $L^{2}(\partial\Omega,d\mu)$, in other words $H^{2}(\partial\Omega)$ is a Hilbert subspace. The Poisson integral $Pf$, $Pf(z)=\int_{\partial\Omega}P(z,w)f(w)d\mu(w)$, is a holomorphic extension of the function $f\in\ H^{2}(\partial\Omega)$ on $\Omega$.
\begin{thm}
The monomials $\{z^{\alpha}\}$ span $H^{2}(\partial\Omega)$.
\label{s2}
\end{thm}
\begin{proof}
For any multi-indices $\alpha$, $z^{\alpha}$ is holomorphic on $\Omega$ and continuous on $\overline{\Omega}$. To prove the completeness we need to show the span of the functions $\{z^{\alpha}\}$ is dense in $A(\Omega)$ with respect to the uniform topology on $\partial\Omega$. The subalgebra of $C(\partial\Omega)$ generated by $\{z^{\alpha}\}$ and $\{\bar{z}^{a}\}$ separates points, contains 1. It is also self-adjoint, therefore Stone-Weierstrass Theorem implies that the closed sub-algebra generated by $\{z^{\alpha}\}$, $\{\bar{z}^{\alpha}\}$ is dense in $A(\Omega)$. Since $\Omega$ is complete Reinhardt then for $f\in\ A(\Omega)$ the functions $\{f_{r}\}_{0\leq r<1}$, $f_{r}(z)=f(rz)$, are holomorphic and uniformly bounded on $\overline{\Omega}$ and continuity of $f$ on $\overline{\Omega}$ implies that $\lim_{r\rightarrow 1}f_{r}(z)=f(z)$ for $z\in\partial\Omega$. Let $\sum_{\beta}c_{\beta}z^{\beta}$ be the power series expansion of $f$ around the origin, therefore $\sum_{\beta}c_{\beta}r^{\beta}z^{\beta}$ uniformly converges to $f_{r}(z)$ on $\overline{\Omega}$ when $0\leq r<1$. So for any nonzero multi-indices $\alpha$ we have,
\begin{equation}
\begin{split}
(f,\bar{z}^{\alpha})&=\int_{\partial\Omega}f(z)z^{\alpha}d\mu(z)=\lim_{r\rightarrow 1}\int_{\partial\Omega}f_{r}(z)z^{\alpha}d\mu(z)\\
&=\lim_{r\rightarrow 1}\int_{\partial\Omega}\sum_{\beta}c_{\beta}(rz)^{\alpha+\beta}d\mu(z)=\lim_{r\rightarrow1}\sum_{\beta}c_{\beta}\int_{\partial\Omega}(rz)^{\alpha+\beta}d\mu(z)=0.\\
\end{split}
\label{s3}
\end{equation}
So the monomials $\{\bar{z}^{\alpha}\}$ are orthogonal to $A(\Omega)$ when $\alpha\neq 0$.
\end{proof}
\begin{prop}
For each fixed $z\in\Omega$, the functional
\begin{equation}
\phi_{z}:H^{2}(\partial\Omega)\rightarrow\mathbb{C},\:\: \phi_{z}(f)=Pf(z),
\label{s4}
\end{equation}
is a linear continuous functional on $H^{2}(\partial\Omega)$ where $Pf(z)$ is the Poisson integral of the function $f$.
\end{prop}
\begin{proof}
Let $\{f_{j}\}_{j=1}^{\infty}$ be a sequence of functions in $H^{2}(\partial\Omega)$ that converges to $f$ in $L^{2}(\partial\Omega,d\mu)$ thus
\begin{equation}
\begin{split}
|Pf(z)-Pf_{j}(z)|&=|\int_{\partial\Omega}P(z,w)f(w)d\mu(w)-\int_{\partial\Omega}P(z,w)f_{j}(w)d\mu(w)|\\
&\leq\int_{\partial\Omega}|P(z,w)f(w)-P(z,w)f_{j}(w)|d\mu(w)\\
&\leq(\int_{\partial\Omega}|P(z,w)|^{2}d\mu(w))^{1/2}(\int_{\partial\Omega}|f(w)-f_{j}(w)|^{2}d\mu(w))^{1/2}\\
&\leq C||f-f_{j}||_{L^{2}(\partial\Omega,d\mu)},\\
\end{split}
\end{equation}
where $P(z,w)$ is the Poisson kernel on $\Omega$.
\end{proof}
\begin{lem}
Let $K\subset\Omega$ be a compact set. There is a constant $C_{K}$ depending on $K$, such that
\begin{equation}
\sup_{z\in K}|Pf(z)|\leq C_{K}||f||_{L^{2}(\partial\Omega,d\mu)}\:\:\text{for all f}\in H^{2}(\partial\Omega).
\end{equation}
\label{s5}
\end{lem}
\begin{proof}
\begin{equation}
\begin{split}
|Pf(z)|=|\int_{\partial\Omega}P(z,w)f(w)d\mu(w)|&\leq ||P(z,.)||_{L^{2}(\partial\Omega)}||f||_{L^{2}(\partial\Omega)}\\
&\leq C_{K}||f||_{L^{2}(\partial\Omega)}.\\
\end{split}
\end{equation}
\end{proof}
The Riesz representation theorem implies that there is a function $k_{z}\in H^{2}(\partial\Omega)$ that represents the linear functional $\phi_{z}$, $\phi_{z}(f)=(f,k_{z})$. We define the Szeg\"{o} kernel $S(z,w)$ by $S(z,w)=\overline{k_{z}(w)}$ for $z\in\Omega$, $w\in\partial\Omega$. To be more precise, $S(z,w)$ is the reproducing kernel of the projection map,
\begin{equation}
Pf(z)=(f,k_{z})=\int_{\partial\Omega}f(w)\overline{k_{z}(w)}d\mu(w)=\int_{\partial\Omega}f(w)S(z,w)d\mu(w),
\label{s6}
\end{equation}
for all $z\in\Omega$.
\begin{lem}
The Szeg\"{o} kernel $S(z,w)$ is conjugate symmetric, $S(z,w)=\overline{S(w,z)}$ for $z,w\in\Omega$.
\label{s7}
\end{lem}
\begin{proof}
For each fixed $w\in\Omega$ we have $\overline{S(w,.)}=k_{w}(.)\in H^{2}(\partial\Omega)$. Hence
\begin{equation}
\begin{split}
\overline{S(w,z)}=P\overline{S(w,.)(z)}&=\int_{\partial\Omega}S(z,y)\overline{S(w,y)}d\mu(y)\\
&=\overline{{\int_{\partial\Omega}S(w,y)\overline{S(z,y)}d\mu(y)}}\\
&=\overline{\overline{S(z,w)}}=S(z,w).
\end{split}
\label{s8}
\end{equation}
\end{proof}
The Szeg\"{o} kernel is unique in the sense that is conjugate symmetric, reproduces $H^{2}(\partial\Omega)$ and holomorphic in the first variable. Since $H^{2}(\partial\Omega)$ is a separable Hilbert space spanned by monomials, so there is a complete orthonormal basis $\{p_{j}\}_{j=0}^{\infty}$ of polynomials for $H^{2}(\partial\Omega)$ with respect to the measure $d\mu$.
\begin{lem}
The series $\sum_{j=0}^{\infty}p_{j}(z)\overline{p_{j}(w)}$ converges uniformly on any compact set $K\times K\subset\Omega\times\Omega$.
\label{s9}
\end{lem}
\begin{proof}
Every element $f\in H^{2}(\partial\Omega)$ has a unique representation, $f=\sum_{j=0}^{\infty}a_{j}p_{j}$, where $\sum_{j=0}^{\infty}|a_{j}|^{2}=||f||^{2}_{L^{2}(\partial\Omega,d\mu)}$. Therefore with respect to the new representation, the linear functional $\phi_{z}$ is
\begin{equation}
\begin{split}
&\phi_{z}:l^{2}\rightarrow\mathbb{C},\\
&(\{a_{j}\})\rightarrow\sum_{j=0}^{\infty}a_{j}p_{j}(z)=(\{a_{j}\},\{p_{j}(z)\}).\\
\end{split}
\end{equation}
So by using Riesz-Fischer Theorem
\begin{equation}
\begin{split}
\sum_{j=0}^{\infty}|p_{j}(z)|^{2}&=\sup_{||\{a_{j}\}||_{l^{2}}=1}|(\{a_{j}\},\{p_{j}(z)\})|^{2}\\
&=\sup_{||\{a_{j}\}||_{l^{2}}=1}|\sum_{j=0}^{\infty}a_{j}p_{j}(z)|^{2}\\
&=\sup_{||f||_{L^{2}(\partial\Omega,d\mu)}=1}|f(z)|^{2}\leq C_{K}^{2}.
\end{split}
\end{equation}
Last inequality follows from the Lemma \eqref{s5}. So the series $\sum_{j=0}^{\infty}|p_{j}(z)|^{2}$ uniformly converges on $K$. Hence if we choose
$N$ big enough such that
$$\sum_{j=m+1}^{n}|p_{j}(z)|^{2}<\epsilon\:\:\text{for}\:m,n>N,$$
then we have
\begin{equation}
\begin{split}
(\sum_{j=m+1}^{n}|p_{j}(z)||p_{j}(w)|)^{2}&\leq(\sum_{j=m+1}^{n}|p_{j}(z)|^{2})(\sum_{j=0}^{\infty}|p_{j}(w)|^{2})\\
&\leq \epsilon C_{K}<.\\
\end{split}
\end{equation}
Therefore the series $\sum_{j=0}^{\infty}p_{j}(z)\overline{p_{j}(w)}$ is uniformly Cauchy on $K\times K$.
\end{proof}
\begin{thm}
The series $\sum_{j=0}^{\infty}p_{j}(z)\overline{p_{j}(w)}$ extends to $(\overline{\Omega}\times\Omega)\cup(\Omega\times\overline{\Omega})$ almost everywhere.
\label{s10}
\end{thm}
\begin{proof}
For $w\in\Omega$ we already showed that $(\sum_{j=0}^{\infty}|p_{j}(w)|^{2})$ is finite. Therefore the function $\sum_{j=0}^{\infty}\overline{p_{j}(w)}p_{j}$ belongs in $H^{2}(\partial\Omega)$, so $\sum_{j=0}^{\infty}\overline{p_{j}(w)}p_{j}$ is holomorphic on $\Omega$ and extends to $\overline{\Omega}$ almost everywhere. Hence the series $\sum_{j=0}^{\infty}p_{j}(z)\overline{p_{j}(w)}$ is bounded almost everywhere on $\overline{\Omega}\times\Omega$ and similarly on $\Omega\times\overline{\Omega}$.
\end{proof}
\begin{thm}
The Szeg\"{o} kernel $S(z,w)$ is equal to the $\sum_{j=0}^{\infty}p_{j}(z)\overline{p_{j}(w)}$.
\label{s11}
\end{thm}
\begin{proof}
 The sum $\sum_{j=0}^{\infty}p_{j}(z)\overline{p_{j}(w)}$ is conjugate symmetric and holomorphic in the first variable for $z\in\Omega$, so to complete the proof we require to show the reproducing property of the $\sum_{j=0}^{\infty}p_{j}(z)\overline{p_{j}(w)}$. For any arbitrary $f\in H^{2}(\partial\Omega)$, $||f||_{L^{2}(\partial\Omega,d\mu)}=\sum_{j=0}^{\infty}|(f,p_{j})|^{2}<\infty$ and the partial sums $\sum_{j=0}^{N}(f,p_{j})p_{j}(z)$ are holomorphic and converge uniformly on any compact subset of $\Omega$. So the sum $\sum_{j=0}^{\infty}(f,p_{j})p_{j}$ is holomorphic on $\Omega$, and for arbitrary $z\in\Omega$ we have
\begin{equation}
\begin{split}
\sum_{j=0}^{\infty}(f,p_{j})p_{j}(z)&=\lim_{n\rightarrow\infty}\sum_{j=0}^{n}(f,p_{j})p_{j}(z)\\
&=\lim_{n\rightarrow\infty}\sum_{j=0}^{n}p_{j}(z)\int_{\partial\Omega}f(w)\overline{p_{j}(w)}d\mu(w)\\
&=\lim_{n\rightarrow\infty}\int_{\partial\Omega}\sum_{j=0}^{n}p_{j}(z)f(w)\overline{p_{j}(w)}d\mu(w)\\
&=\int_{\partial\Omega}\sum_{j=0}^{\infty}p_{j}(z)\overline{p_{j}(w)}f(w)d\mu(w),\\
\end{split}
\end{equation}
where the last two equations follow from the Theorem \eqref{s10} and Lebesgue dominated convergence Theorem. So
\begin{equation}
\int_{\partial\Omega}(\sum_{j=0}^{\infty}p_{j}(z)\overline{p_{j}(w)})f(w)d\mu(w)=\sum_{j=0}^{\infty}(f,p_{j})p_{j}(z),
\end{equation}
that implies the integral $\int_{\partial\Omega}(\sum_{j=0}^{\infty}p_{j}(z)\overline{p_{j}(w)})f(w)d\mu(w)$ is a holomorphic extension of $f$ to $\Omega$. Therefore $\sum_{j=0}^{\infty}p_{j}(z)\overline{p_{j}(w)}$ reproduces $H^{2}(\partial\Omega)$. Since the Szeg\"{o} kernel is unique, it implies that $S(z,w)=\sum_{j=0}^{\infty}p_{j}(z)\overline{p_{j}(w)}$.
\end{proof}
\begin{prop}
If $f\in\L^{2}(\partial\Omega,d\mu)$ then $\int_{\partial\Omega}f(w)S(z,w)d\mu(w)$ belongs to $H^{2}(\partial\Omega)$.
\label{s12}
\end{prop}
\begin{proof}
Functions $\{p_{j}\}_{j=0}^{\infty}$ form an orthonormal basis for $H^{2}(\partial\Omega)\subset\L^{2}(\partial\Omega,d\mu)$, so
\begin{equation}
\sum_{j=0}^{\infty}|(f,p_{j})|^{2}\leq ||f||_{L^{2}(\partial\Omega,d\mu)}<\infty\:\:\text{for}\:f\in\L^{2}(\partial\Omega,d\mu).
\end{equation}
This means $\sum_{j=0}^{\infty}(f,p_{j})p_{j}\in H^{2}(\partial\Omega)$, so by using Theorem \eqref{s11} we have
\begin{equation}
\int_{\partial\Omega}f(w)S(z,w)d\mu(w)=\sum_{j=0}^{\infty}(f,p_{j})p_{j},
\end{equation}
that implies $\int_{\partial\Omega}f(w)S(z,w)d\mu(w)\in H^{2}(\partial\Omega)$.
\end{proof}
 Proposition \eqref{s12} introduces a new representation of the Szeg\"{o} kernel. We can think of $S(z,w)$  as the kernel of the orthogonal projection map from $L^{2}(\partial\Omega,d\mu)$ to $H^{2}(\partial\Omega)$,
\begin{equation}
\begin{split}
&\Pi:L^{2}(\partial\Omega,d\mu)\rightarrow H^{2}(\partial\Omega),\\
&\Pi(f)(z)=\int_{\partial\Omega}f(w)S(z,w)d\mu(w)=\sum_{j=0}^{\infty}(f,p_{j})p_{j}(z).\\
\end{split}
\end{equation}
Let's define $H_{K}(\partial\Omega)$ to be the closed subspace of $H^{2}(\partial\Omega)$ spanned by $\{z^{\alpha}\}$ for $|\alpha|=K$. Since $\Omega$ is a Reinhardt domain then $H_{K}\cap H_{K'}=\{0\}$ for $K\neq K'$ and monomials span $H^{2}(\partial\Omega)$ by using Theorem \eqref{s2}. So
\begin{equation}
H^{2}(\partial\Omega)=\bigoplus\sum_{K=0}^{\infty}H_{K}(\partial\Omega).
\end{equation}
We define the orthogonal projection map,
\begin{equation}
\begin{split}
&\Pi_{K}:L^{2}(\partial\Omega,d\mu)\rightarrow H_{K}(\partial\Omega),\\
&\Pi_{K}(f)(z)=\sum_{K_{j}\in I_{K}}(f,p_{K_{j}})p_{K_{j}}(z),\\
\end{split}
\end{equation}
where $\{p_{K_{j}}\}_{I_{K}}$ is the subset of the orthonormal basis $\{p_{j}\}_{j=0}^{\infty}$ that spans $H_{K}(\partial\Omega)$. Therefore
\begin{equation}
\begin{split}
&\Pi=\bigoplus\sum_{K=0}^{\infty}\Pi_{K}\:\:\text{and consequently},\\
&S(z,w)=\sum_{K=0}^{\infty}\Pi_{K}(z,w).\\
\end{split}
\end{equation}
\begin{thm}
Let $\Pi_{K}(z,w)$ be the conjugate symmetric reproducing kernel for the projection map $\Pi_{K}$, then
\begin{equation}
\Pi_{K}(z,w)=\frac{1}{2\pi}\int_{0}^{2\pi}e^{-iK\theta}S(e^{i\theta}z,w)d\theta.
\end{equation}
\label{s13}
\end{thm}
\begin{proof}
 The Szeg\"{o} kernel $S(z,w)$ is conjugate symmetric and holomorphic in the first variable, so if we let $\tilde{\Pi}_{K}(z,w)=\frac{1}{2\pi}\int_{0}^{2\pi}e^{-iK\theta}S(e^{i\theta}z,w)d\theta$ then $\tilde{\Pi}_{K}(z,w)$ satisfies the same properties. For any monomial $z^{\alpha}$ we have,
\begin{equation}
\begin{split}
\int_{\partial\Omega}\tilde{\Pi}_{K}(z,w)w^{\alpha}d\mu(w)&=\frac{1}{2\pi}\int_{\partial\Omega}\int_{0}^{2\pi}e^{-iK\theta}S(e^{i\theta}z,w)d\theta w^{\alpha}d\mu(w)\\
&=\frac{1}{2\pi}\int_{0}^{2\pi}e^{-iK\theta}\int_{\partial\Omega}S(e^{i\theta}z,w)w^{\alpha}d\mu(w)d\theta\\
&=\frac{1}{2\pi}\int_{0}^{2\pi}e^{-iK\theta}(e^{i\theta}z)^{\alpha}d\theta\\
&=\frac{z^{\alpha}}{2\pi}\int_{0}^{2\pi}e^{-iK\theta}e^{i|\alpha|\theta}d\theta.\\
\end{split}
\end{equation}
Therefore, if $|\alpha|=K$ then
\begin{equation}
\int_{\partial\Omega}\tilde{\Pi}_{K}(z,w)w^{\alpha}d\mu(w)=z^{\alpha},
\end{equation}
and for $|\alpha|\neq K$
\begin{equation}
\int_{\partial\Omega}\tilde{\Pi}_{K}(z,w)w^{\alpha}d\mu(w)=0,
\end{equation}
so $\tilde{\Pi}_{K}(z,w)$ is the reproducing kernel of $\Pi_{K}$ that is conjugate symmetric. So by using the uniqueness property of the Szeg\"{o} kernel, $\tilde{\Pi}_{K}(z,w)=\Pi_{K}(z,w)$.
\end{proof}
\subsection{Boutet de Monvel-Sj\"{o}strand Theorem and Partial Szeg\"{o} kernels}

\begin{thm}
Let $S(x,y)$ be the Szeg\"{o} kernel of the boundary X of a bounded strictly pseudo-convex domain $\Omega$ in a complex manifold. Then there exists a symbol $s\in S^{n}(X\times X\times  \Re^{+})$ of the type $s(x,y,t)\sim \sum_{k=0}^{\infty} t^{m-k}s_{k}(x,y)$ so that,

   $$S(x,y)=\int_{0}^{\infty} e^{it\psi(x,y)}s(x,y,t)dt,$$

   where the phase $\psi\in C^{\infty}(X\times X)$ is determined by the following properties:

   1) $\psi(x,x)=\frac{\rho(x)}{i}$ where $\rho$ is the defining function of X.

   2) $\bar{\partial}_{x}\psi$ and $\partial_{y}\psi$ vanish to infinite order along diagonal.

   3) $\psi(x,y)=-\bar{\psi}(y,x).$

   \end{thm}
 The integral is defined as a complex oscillatory integral and is regularized by taking the principal value. So our goal is to find asymptotic expansion for $\Pi_{K}(z,z)$ by using above Theorem. Theorem \eqref{s13} implies
\begin{equation}
\begin{split}
\Pi_{K}(z,z)&=\frac{1}{2\pi}\int_{0}^{2\pi}e^{-iK\theta}S(e^{i\theta}z,z)d\theta\\
&=\frac{1}{2\pi}\int_{0}^{\infty}\int_{0}^{2\pi}e^{-i K\theta}e^{it\psi(e^{i\theta }z,z)}s(e^{i\theta}z,z,t)d\theta dt.
\end{split}
\end{equation}

For simplicity we let $s(e^{i\theta}z,z,t):=\frac{1}{2\pi}s(e^{i\theta}z,z,t)$. By using the change of variable
$$t\rightarrow Kt\:\:,\:\: \phi(t,\theta;z,z)=\theta-t\psi(r_{\theta} z,z),$$
we have
\begin{equation}
\Pi_{K}(z,z)=K\int_{0}^{\infty}\int_{0}^{2\pi}e^{-iK\phi(\theta,t;z,z)}s(r_{\theta}z,z,Kt)d\theta dt.
\label{c}
\end{equation}
Also we have
\begin{equation}
Im\psi(z,w)\geq c(d(z,X)+d(w,X)+|z-w|^2)+O(|z-w|^3),
\label{d}
\end{equation}
 where $c$ is a positive constant. This results in $Im\psi(z,w)\geq0$. We want to give an asymptotic expansion for \eqref{c} by using stationary phase method. For this purpose we need to consider phase function, hence first step is to find the critical point of the phase function.

\begin{lem}
The phase function $\phi(\theta,t;z,z)=\theta-t\psi(r_{\theta} z,z)$ has only one critical point, $(0,\frac{1}{d'\rho(z)\cdot z})$.
\end{lem}
\begin{proof}
If $\frac{\partial\phi}{\partial t}=0$ then $\psi(r_{\theta}z,z)=0$. Now by using  \eqref{d},
\begin{equation}
\psi(r_{\theta}z,z)=0\leftrightarrow r_{\theta} z=z \leftrightarrow\theta=0.
\end{equation}
Next by taking derivative respect to $\theta$ we have
\begin{equation}
\frac{\partial\phi}{\partial\theta}=1-te^{i\theta}\sum_{i=0}^{i=m}\frac{\partial_{x}\psi(r_{\theta}z,z)}{\partial x_{i}}z_{i},
\end{equation}
next we plug in $\theta=0$
\begin{equation}
\frac{\partial\phi}{\partial\theta}=1-td'\rho(z)\cdot z\:\:\text{for}\:\theta=0.
\end{equation}
We know $\Omega$ is a strictly pseudoconvex domain, so the holomorphic tangent plane at the point $z\in X$ doesn't go through the domain. Consequently
\begin{equation}
0\notin T_{z}^{h}X=\{w\in\mathbb{C}^{m+1}: d'\rho(z)\cdot (z-w)=0\}\rightarrow d'\rho(z)\cdot z\neq 0,
\end{equation}
that implies
\begin{equation}
1-td'\rho(z)\cdot z=0 \rightarrow t_{0}=\frac{1}{d'\rho(z)\cdot z}.
\end{equation}
 It is also a nondegenerate critical point because,
\begin{equation}
|\phi''(0,t_{0})|=|\begin{pmatrix}0&\frac{1}{t_{0}}\\\frac{1}{t_{0}}&\frac{\partial^{2}\phi}{\partial\theta^{2}} \end{pmatrix}|=-\frac{1}{t_{0}^{2}}<0.
\end{equation}
\end{proof}
\begin{thm}
For $z=(z_{0},\dots,z_{m})\in X\cap\mathbb({C}^{*})^{m+1}$ we have
\begin{equation}
\Pi_{K}(z,z)=s_{0}(z,z)t_{0}(Kt_{0})^{m}+R_{K,0},
\label{e}
\end{equation}
such that $|R_{K,0}|\leq C_{0}K^{m-1}$, and $C_{0}$ depends on $X,\psi,z$ and $s_{0}$ is the first term of the symbol $s(z,z,t)$ and $t_{0}$ is equal to $\frac{1}{d'\rho(z)\cdot z}$.
\end{thm}
\begin{proof}
By using inequality \eqref{d} we see that the imaginary part of $-\phi(t,\theta)$ that is equal to the imaginary part of $t\psi(r_{\theta}z,z)$ is positive everywhere on $ [0,2\pi]\times[0,+\infty)$ except at the critical point. If we choose $K_{\epsilon}$ be a compact set in $[0,2\pi]\times[0,+\infty)$ that includes critical point $(0,t_{0})$ and we let $K_{\epsilon}^{c}=[0,2\pi]\times[0,+\infty)-K_{\epsilon}$ then
\begin{equation}
K\int_{K_{\epsilon}^{c}}e^{-iK\phi(\theta,t;z,z)}s(r_{\theta}z,z,Kt)d\theta dt=O(K^{-\infty}).
\end{equation}
Next by using \eqref{c} we will have
\begin{equation}
\begin{split}
\Pi_{K}(z,z)=&K\int_{K_{\epsilon}}e^{-iK\phi(\theta,t;z,z)}s(r_{\theta}z,z,Kt)d\theta dt+\\
&K\int_{K_{\epsilon}^{c}}e^{-iK\phi(\theta,t;z,z)}s(r_{\theta}z,z,Kt)d\theta dt.\\
\end{split}
\end{equation}
To compute the first term in the last equation we use stationary phase method. As we already proved our critical point is nondegenerate and here we are taking integral over the compact set $K_{\epsilon}$ which includes $(0,t_{0})$. By using Theorem (7.7.5) from \cite{Ho},
\begin{equation}
\begin{split}
&K\int_{K_{\epsilon}}e^{-iK\phi(\theta,t;z,z)}s(r_{\theta}z,z,Kt)d\theta dt\\ &\sim\frac{K}{\sqrt{\frac{K\phi''(0,t_{0})}{2\pi i}}}\sum_{j,k=0}^{\infty}K^{m-j-k}L_{j}(t^{m-k}s_{k}(r_{\theta}z,z)),\\
 &which\\
 &L_{j}(a)=\sum_{\nu-\mu=j}\sum_{2\nu\geq 3\mu}\frac{i^{-j}2^{-\nu}}{\mu!\nu!}\big<\phi''(0,t_{0})^{-1}D,D\big>(g_{(0,t_{0})}^{\mu}(t,\theta)a).\\
\end{split}
\label{e1}
\end{equation}
In this equation $g_{(0,t_{0})}$ is equal to the third order reminder of $\phi(\theta,t)$ at $(0,t_{0})$ and in the left hand side you can see that if $j=0,k=0$ then we will get the highest power of $K$. By looking at the definition of $L_{j}$ we have $L_{0}(t^{m}s_{0}(r_{\theta} z,z))=t_{0}^{m}s_{0}(z,z)$, and by using the stationary phase Theorem from \cite{Ho}:
\begin{equation}
\begin{split}
|\Pi_{K}(z,z)-t_{0}K^{m}L_{0}(t^{m}s_{0}(r_{\theta}z,z))|&=|\Pi_{K}(z,z)-K^{m}t_{0}^{m}s_{0}(z,z)t_{0}|\\
&\leq K^{m-1}CM,\\
\end{split}
\end{equation}
where $M=\sum_{|\alpha|\leq2}||D^{\alpha}s||_{\infty}.$
\end{proof}
For the next step we need to find asymptotic expansion for the derivatives of $\Pi_{K}(z,z)$ by using \eqref{c}. For that purpose we introduce some notations that help us to understand the derivatives of $\Pi_{K}(z,z)$. We know that $s(x,y,t)$ is a smooth function on $X\times\mathbb{R}$, but we don't know about the behavior of $s(x,y,t)$ on the neighborhood of $X$ in $\mathbb{C}^{m+1}$.
 So we can only use \eqref{c} for computing derivatives of $\Pi_{k}(z,z)$ in real tangential directions.
  Now let's talk more about the real tangent plane on $X$ at point $z=(z_{0},\dots,z_{m})\in X$.
Reinhardt property of the $\Omega$ implies that
 \begin{equation}
 (e^{i\theta_{0}}z_{0},\dots,z_{m})\in X\:\:\text{{for}}\:\:\theta_{0}\in[0,2\pi],
 \end{equation}
so we have
\begin{equation}
\frac{\partial}{\partial\theta_{0}}=(iz_{0},\dots,z_{m}),
\end{equation}
 and similarly we can define $\frac{\partial}{\partial\theta_{j}}$.
\begin{lem}
If $f:\mathbb{C}^{m+1}\rightarrow \mathbb{C}$ is an anti holomorphic function then
\begin{equation}
D_{\theta_{j}}f(z)=-i\bar{z}_{j}\frac{\partial f}{\partial\bar{z}_{j}}
\end{equation}
\label{f1}
\end{lem}
Now we introduce some notations to simplify computations. Let
\begin{equation}
\begin{split}
&\alpha=(\alpha_{0},\dots,\alpha_{m}),\beta=(\beta_{0},\dots,\beta_{m}),\\
&\gamma_{i}=(\gamma_{i,0},\dots,\gamma_{i,m}),\{k_{i}\},\\
&which\\&\alpha_{i},\beta_{i},\gamma_{i,j},k_{i}\in\{0,1,2,\dots\},\\
&I_{\alpha}=\{l=(\beta,\{\gamma_{i}\},\{k_{i}\}):\sum k_{i}\gamma_{i}+\beta=\alpha\}.\\
\end{split}
\label{f}
\end{equation}
For any multi indices $\alpha=(\alpha_{0},\dots,\alpha_{m})$ we define:
\begin{equation}
D^{\alpha}= D^{\alpha_{m}}_{\theta_{m}}\dots D^{\alpha_{0}}_{\theta_{0}}.
\end{equation}
If $l\in I_{\alpha}$ then we define
\begin{equation}
Z_{l}(f,g)=\Pi(D^{\gamma_{i}}f)^{k_{i}}(D^{\beta}g).
\label{g}
\end{equation}
If we let $l_{0}=(\beta,\{\gamma_{i}\},\{k_{i}\})$ such that $\beta=(0,\dots,0)$, $\gamma_{0}=(1,0,\dots,0)$,\\
$\dots,$ $\gamma_{m}=(0,\dots,1), k_{0}=\alpha_{0},\dots,k_{m}=\alpha_{m}$ then
\begin{equation}
\begin{split}
Z_{l_{0}}(i\psi(r_{\theta}z,z),s_{0}(r_{\theta}z,z))|_{\theta=0}
&=\Pi(iD^{\gamma_{i}}_{y}\psi(r_{\theta}z,z))^{\alpha_{i}}s_{0}(r_{\theta}z,z)|_{\theta=0}\\
&=\Pi(i\frac{\partial_{y}\psi(r_{\theta}z,z)}{\partial\bar{z}_{i}}(-i\bar{z_{i}}))^{\alpha_{i}}s_{0}(r_{\theta}z,z)|_{\theta=0}\\
&=(\frac{\partial\rho}{\partial\bar{z}_{0}}\dots\frac{\partial\rho}{\partial\bar{z}_{m}})^{\alpha}(-i\bar{z})^{\alpha}s_{0}(z,z).\\
\end{split}
\end{equation}
\begin{lem}
There are constants $c_{l}$ only depend on $l,\alpha$ such that 
$$D^{\alpha}(e^{f}g)=e^{f}\sum_{c_{l}\in I_{\alpha}}c_{l}Z_{l}(f,g).$$
\label{h}
\end{lem}
Now by using lemma \eqref{h} we have this result:
\begin{equation}
\begin{split}
D_{y}^{\alpha}(e^{-iK\phi }s(r_{\theta}z,z,Kt))
&=\sum_{c_{l}\in I_{\alpha}}c_{l}e^{-iK\phi}Z_{l}(-iK\phi,s(r_{\theta}z,z,Kt))\\
&=\sum_{k=0}^{\infty}\sum_{l\in I_{\alpha}}c_{l}e^{-iK\phi}(Kt)^{\sum k_{i}}Z_{l}(i\psi,s_{k})(Kt)^{m-k}\\
&=\sum_{k=0}^{\infty}\sum_{l\in I_{\alpha}}c_{l}e^{-iK\phi}(Kt)^{m+\sum k_{i}-k}Z_{l}(i\psi,s_{k}).\\
\end{split}
\end{equation}
\begin{thm}
If $z=(z_{0},\dots,z_{m})\in X\cap(\mathbb{C}^{*})^{m+1}$ then there is a constant $C_{\alpha}$ that only depends on $z,\alpha,\psi,X$ such that:
\begin{equation}
D_{y}^{\alpha}\Pi_{K}(z,z)=s_{0}(z,z)t_{0}(Kt_{0})^{m+|\alpha|}
(\frac{\partial\rho}{\partial\bar{z}_{0}}\dots\frac{\partial\rho}{\partial\bar{z}_{m}})^{\alpha}(-i\bar{z})^{\alpha}+R_{K,\alpha},
\end{equation}
 $|R_{K,\alpha}|\leq C_{\alpha}K^{m+|\alpha|-1}.$
\label{i}
\end{thm}
\begin{proof}
If we use equation \eqref{c} and lemma \eqref{h} then
\begin{equation}
\begin{split}
D_{y}^{\alpha}\Pi_{K}(z,z)
&=D_{y}^{\alpha}K\int_{0}^{\infty}\int_{0}^{2\pi}e^{-iK\phi(\theta,t;z,z)}s(r_{\theta}z,z,Kt)d\theta dt\\
&=K\int_{0}^{\infty}\int_{0}^{2\pi}D_{y}^{\alpha}(e^{-iK\phi(\theta,t;z,z)}s(r_{\theta}z,z,Kt))d\theta dt\\
&=K\sum_{l\in I_{\alpha}}c_{l}\int_{0}^{\infty}\int_{0}^{2\pi}e^{-iK\phi}Z_{l}(-iK\phi,s(r_{\theta} z,z,Kt))d\theta dt\\
&=K\sum_{l\in I_{\alpha}}c_{l}\int_{0}^{\infty}\int_{0}^{2\pi}e^{-iK\phi}Z_{l}(iKt\psi,s(r_{\theta} z,z,Kt))d\theta dt\\
&=K\sum_{l\in I_{\alpha}}c_{l}\int_{0}^{\infty}\int_{0}^{2\pi}e^{-iK\phi}(Kt)^{\sum k_{i}}Z_{l}(i\psi,s(r_{\theta} z,z,Kt))d\theta dt\\
&\sim\sum_{l\in I_{\alpha}}c_{l}\frac{K}{\sqrt{|K\phi''(0,t_{0})/2\pi i|}}\sum_{k,j=0}^{\infty}K^{-j}L_{j}((Kt)^{m+\sum k_{i}-k}Z_{l}(i\psi,s_{k})).\\
\end{split}
\end{equation}
If we look at in the series then the highest degree of $K$ happens whenever $l=l_{0},k=j=0$. In this case $k_{i}=\alpha_{i},c_{l_{0}}=1$ and by using equation \eqref{g} and  using Theorem(7.7.5) from \cite{Ho} we will get this result,
\begin{equation}
\begin{split}
&|D_{y}^{\alpha}\Pi_{K}(z,z)-(Kt_{0})^{m+|\alpha|}
(\frac{\partial\rho}{\partial\bar{z}_{0}}\dots\frac{\partial\rho}{\partial\bar{z}_{m}})^{\alpha}(-i\bar{z})^{\alpha}s_{0}(z,z)t_{0}|\\
&\leq K^{m+|\alpha|-1}C \sum_{l\in I_{\alpha}}\sum_{|\beta|\leq 2 }||D^{\beta}Z_{l}(i\psi,s)||_{\infty}.\\
\end{split}
\end{equation}
If we let $M=C \sum_{l\in I_{\alpha}}\sum_{|\beta|\leq 2 }||D^{\beta}Z_{l}(-i\psi,s)||_{\infty}$ then
\begin{equation}
|D_{y}^{\alpha}\Pi_{K}(z,z)-(Kt_{0})^{m+|\alpha|}
(\frac{\partial\rho}{\partial\bar{z}_{0}}\dots\frac{\partial\rho}{\partial\bar{z}_{m}})^{\alpha}(-i\bar{z})^{\alpha}s_{0}(z,z)t_{0}|\leq M K^{m+|\alpha|-1},
\end{equation}
where $M$ is a constant that only depends on $\psi,\rho$ and their partial derivatives.
 So I can tell,
\begin{equation}
D_{y}^{\alpha}\Pi_{K}(z,z)=(Kt_{0})^{m+|\alpha|}
(\frac{\partial\rho}{\partial\bar{z}_{0}}\dots\frac{\partial\rho}{\partial\bar{z}_{m}})^{\alpha}(-i\bar{z})^{\alpha}s_{0}(z,z)t_{0}+R_{K,\alpha},
\end{equation}
where
$|R_{K,\alpha}|\leq M K^{m+|\alpha|-1}.$
\end{proof}
An upper bound for $D^{\alpha}_{y}\Pi_{K}(z,z)$ is,
\begin{equation}
\begin{split}
D^{\alpha}_{y}\Pi_{K}(z,z)
&=(Kt_{0})^{m+|\alpha|}
(\frac{\partial\rho}{\partial\bar{z}_{0}}\dots\frac{\partial\rho}{\partial\bar{z}_{m}})^{\alpha}(-i\bar{z})^{\alpha}s_{0}(z,z)t_{0}+R_{K,\alpha}\\
&\leq (Kt_{0})^{m+|\alpha|}
(\frac{\partial\rho}{\partial\bar{z}_{0}}\dots\frac{\partial\rho}{\partial\bar{z}_{m}})^{\alpha}(-i\bar{z})^{\alpha}s_{0}(z,z)t_{0}+M K^{m+|\alpha|-1}\\
&\leq K^{m+|\alpha|}((\frac{\partial\rho}{\partial\bar{z}_{0}}\dots\frac{\partial\rho}{\partial\bar{z}_{m}})^{\alpha}(-i\bar{z})^{\alpha}s_{0}(z,z)t_{0}+M)\\
&\leq K^{m+|\alpha|}M_{\alpha}',\\
\end{split}
\label{j}
\end{equation}
where $M_{\alpha}'$ depends on $M,\rho,z,\alpha$.
\begin{lem}
If f is an anti holomorphic function on $\mathbb{C}^{m+1}$ then
\begin{equation}
\frac{\partial^{\alpha}f}{\partial\bar{z}^{\alpha}}=\frac{1}{(-i\bar{z})^{\alpha}}D^{\alpha}f+\sum_{|\beta|<|\alpha|}e_{\beta}D^{\beta}f,
\end{equation}
where $e_{\beta}$ only depends on $\alpha,\beta,z$.
\label{k}
\end{lem}
\begin{thm}
If $z=(z_{0},\dots,z_{m})\in X\cap(\mathbb{C}^{*})^{m+1}$ then there is a constant $C'_{\alpha}$ such that,
\begin{equation}
\frac{\partial^{\alpha}}{\partial\bar{z}^{\alpha}}\Pi_{K}(z,z)=s_{0}(z,z)t_{0}(Kt_{0})^{m+|\alpha|}
(\frac{\partial\rho}{\partial\bar{z}_{0}}\dots\frac{\partial\rho}{\partial\bar{z}_{m}})^{\alpha}+R'_{K,\alpha},
\label{l}
\end{equation}
and $|R'_{K,\alpha}|\leq C'_{\alpha}K^{m+|\alpha|-1}.$
\end{thm}
\begin{proof}
By using lemma \eqref{k} and Theorem \eqref{i} we have,
\begin{equation}
\begin{split}
&\frac{\partial^{\alpha}}{\partial\bar{z}^{\alpha}}\Pi_{K}(z,z)
=\frac{1}{(\bar{z})^{\alpha}}D^{\alpha}\Pi_{K}(z,z)+\sum_{|\beta|<|\alpha|}e_{\beta}D^{\beta}\Pi_{K}(z,z)\\
&=\frac{1}{(\bar{z})^{\alpha}}(s_{0}(z,z)t_{0}(Kt_{0})^{m+|\alpha|}
(\frac{\partial\rho}{\partial\bar{z}_{0}}\dots\frac{\partial\rho}{\partial\bar{z}_{m}})^{\alpha}(-i\bar{z})^{\alpha}+R_{K,\alpha})
+\sum_{|\beta|<|\alpha|}e_{\beta}D^{\beta}\Pi_{K}(z,z)\\
&=s_{0}(z,z)t_{0}(Kt_{0})^{m+|\alpha|}
(\frac{\partial\rho}{\partial\bar{z}_{0}}\dots\frac{\partial\rho}{\partial\bar{z}_{m}})^{\alpha}+\frac{1}{(-i\bar{z})^{\alpha}}R_{K,\alpha}
+\sum_{|\beta|<|\alpha|}e_{\beta}D^{\beta}_{T}\Pi_{K}(z,z)\\
&=s_{0}(z,z)t_{0}(Kt_{0})^{m+|\alpha|}(\frac{\partial\rho}{\partial\bar{z}_{0}}\dots\frac{\partial\rho}{\partial\bar{z}_{m}})^{\alpha}+R'_{K,\alpha},\\
\end{split}
\end{equation}
where
\begin{equation}
R'_{K,\alpha}=\frac{1}{(\bar{z})^{\alpha}}R_{K,\alpha}+\sum_{|\beta|<|\alpha|}e_{\beta}D^{\beta}\Pi_{K}(z,z).
\end{equation}
Now by using inequality \eqref{j}
\begin{equation}
\begin{split}
R'_{K,\alpha}
&=\frac{1}{(-i\bar{z})^{\alpha}}R_{K,\alpha}+\sum_{|\beta|<|\alpha|}e_{\beta}D^{\beta}\Pi_{K}(z,z)\\
&\leq \frac{1}{(-i\bar{z})^{\alpha}}C_{\alpha}K^{M+|\alpha|-1}+\sum_{|\beta|<|\alpha|}e_{\beta}M'_{\beta}K^{M+|\beta|}\\
&\leq K^{M+|\alpha|-1}( \frac{1}{(-i\bar{z})^{\alpha}}C_{\alpha}+\sum_{|\beta|<|\alpha|}e_{\beta}M'_{\beta})\\
&= K^{M+|\alpha|-1}C'_{\alpha},\\
\end{split}
\end{equation}
where
$$C'_{\alpha}=\frac{1}{(-i\bar{z})^{\alpha}}C_{\alpha}+\sum_{|\beta|<|\alpha|}e_{\beta}M'_{\beta}.$$
\end{proof}
\begin{thm} For any $z=(z_{0},\dots,z_{m})\in X\cap(\mathbb{C^{*}})^{m+1}$ we have
\begin{equation}
\lim_{N\rightarrow\infty}\frac{1}{N^{m+|\alpha|+1}}\frac{\partial^{\alpha}}{\partial\bar{z}^\alpha}S_{N}(z,z)=s_{0}(z,z)(t_{0})^{m+1}\int_{0}^{1}y^{m}
(yt_{0}\frac{\partial\rho}{\partial\bar{z}_{0}}\dots\frac{\partial\rho}{\partial\bar{z}_{m}})^{\alpha}dy.
\end{equation}
\label{m}
\end{thm}
\begin{proof}
\begin{equation}
\begin{split}
&\lim_{N\rightarrow\infty}\frac{1}{N^{m+|\alpha|+1}}\frac{\partial^{\alpha}}{\partial\bar{z}^\alpha}S_{N}(z,z)
=\lim_{N\rightarrow\infty}\frac{1}{N^{m+|\alpha|+1}}\sum_{K=0}^{N}\frac{\partial^{\alpha}}{\partial\bar{z}^\alpha}\Pi_{K}(z,z)\\
&=\lim_{N\rightarrow\infty}\frac{1}{N^{m+|\alpha|+1}}(\sum_{K=0}^{N}(Kt_{0})^{m+|\alpha|}(\frac{\partial\rho}{\partial\bar{z}_{0}}...
\frac{\partial\rho}{\partial\bar{z}_{m}})^{\alpha}s_{0}(z,z)t_{0}+\sum_{K=0}^{N}R^{'}_{K,\alpha})\\ &=(\frac{\partial\rho}{\partial\bar{z}_{0}}\dots\frac{\partial\rho}{\partial\bar{z}_{m}})^{\alpha}s_{0}(z,z)t_{0}
\lim_{N\rightarrow\infty}(\sum_{K=0}^{N}(\frac{Kt_{0}}{N})^{m+|\alpha|}\frac{1}{N}+\sum_{K=0}^{N}\frac{R^{'}_{K,\alpha}}{N^{m+|\alpha|}}\frac{1}{N})\\
&=(\frac{\partial\rho}{\partial\bar{z}_{0}}\dots\frac{\partial\rho}{\partial\bar{z}_{m}})^{\alpha}
s_{0}(z,z)(t_{0})^{m+|\alpha|+1}\int_{0}^{1}y^{m+|\alpha|}dy+0\\
&=s_{0}(z,z)(t_{0})^{m+1}(t_{0}\frac{\partial\rho}{\partial\bar{z}_{0}}...\frac{\partial\rho}{\partial\bar{z}_{m}})^{\alpha}\int_{0}^{1}y^{m+|\alpha|}dy\\
&=s_{0}(z,z)(t_{0})^{m+1}\int_{0}^{1}y^{m}(yt_{0}\frac{\partial\rho}{\partial\bar{z}_{0}}\dots\frac{\partial\rho}{\partial\bar{z}_{m}})^{\alpha}dy.\\
\end{split}
\end{equation}
\end{proof}
For the next step we consider the behavior of the scaling Szeg\"{o} kernel when $N$ goes to infinity. For this purpose we pick a point on the $X$ and we call it $z=(z_{0},\dots,z_{m})$ then we move in the direction of $u=(u_{0},\dots,u_{m})\in \mathbb{C}^{m+1}$. For the simplicity we define,
\begin{equation}
G_{N}(u)=\{\frac{S_{N}(z+\frac{u}{N},z)}{N^{m+1}}\}.
\end{equation}
We want to use Arzela Ascoli Theorem to show that $G_{N}(u)$ uniformly converges on any compact set in $\mathbb{C}^{m+1}$. I should mention that we fix the point $z\in X$.
\begin{lem}
$G_{N}(u)$ is uniformly bounded on $\bar{B}(0,1)\subset\mathbb{C}^{m+1}$.
\label{n}
\end{lem}
\begin{proof}
\begin{equation}
\begin{split}
|G_{N}(u)|
&=|\frac{1}{N^{m+1}}S_{N}(z+\frac{u}{N},z)|
=|\frac{1}{N^{m+1}}\sum_{|J|\leq N}c_{J}(1+\frac{u}{Nz})^{J}z^{J}\bar{z}^{J}|\\
&=|\frac{1}{N^{m+1}}\sum_{|J|\leq N}((1+\frac{u_{0}}{Nz_{0}})^{J_{0}}\dots(1+\frac{u_{m}}{Nz_{m}})^{J_{m}})c_{J}z^{J}\bar{z}^{J}|\\
&\leq\frac{1}{N^{m+1}}\sum_{|J|\leq N}(|1+\frac{u_{0}}{Nz_{0}}|^{J_{0}}\dots|1+\frac{u_{m}}{Nz_{m}}|^{J_{m}})c_{J}z^{J}\bar{z}^{J}\\
&\leq e^{\sum_{i=0}^{m}|\frac{u_{i}}{z_{i}}|}\frac{1}{N^{m+1}}\sum_{|J|\leq N}c_{J}z^{J}\bar{z}^{J}\\
&=e^{\sum_{i=0}^{m}|\frac{u_{i}}{z_{i}}|}\frac{1}{N^{m+1}}S_{N}(z,z).\\
\end{split}
\end{equation}
At the end we have,
\begin{equation}
|G_{N}(u)|\leq e^{\sum_{i=0}^{m}|\frac{u_{i}}{z_{i}}|}\frac{1}{N^{m+1}}S_{N}(z,z).
\end{equation}
By using theorem \eqref{m}, we see that $\frac{1}{N^{m+1}}S_{N}(z,z)$ converges. So there is a positive constant $M$ such that $|\frac{1}{N^{m+1}}S_{N}(z,z)|\leq M$. So
\begin{equation}
|G_{N}(u)|\leq M e^{\sum_{i=0}^{m}|\frac{u_{i}}{z_{i}}|}.
\end{equation}
\end{proof}
\begin{lem}
$\frac{\partial}{\partial{u_{i}}}G_{N}(u)$ is uniformly bounded on $\bar{B}(0,1)\subset\mathbb{C}^{m+1}$ for $i=0,\dots,m$.
\label{n1}
\end{lem}
\begin{proof}
We prove this lemma for $i=0$. Same proof works for $i=1,\dots m$.
\begin{equation}
\begin{split}
|\frac{\partial}{\partial{u_{0}}}G_{N}(u)|
&=|\frac{1}{N^{m+1}}\frac{\partial}{\partial u_{0}}S_{N}(z+\frac{u}{N},z)|
=|\frac{1}{N^{m+2}}\frac{\partial}{\partial z_{0}}S_{N}(z+\frac{u}{N},z)|\\
&=|\frac{1}{N^{m+2}}\sum_{|J|\leq N}c_{J}j_{0}(z_{0}+\frac{u_{0}}{N})^{j_{0}-1}\dots(z_{m}+\frac{u_{m}}{N})^{j_{m}}\bar{z}^{J}|\\
&=|\frac{1}{N^{m+2}}\sum_{|J|\leq N}((1+\frac{u_{0}}{Nz_{0}})^{j_{0}-1}\dots(1+\frac{u_{m}}{Nz_{m}})^{j_{m}})c_{J}z_{0}^{j_{0}-1}...z_{m}^{j_{m}}\bar{z}^{J}|\\
&\leq\frac{1}{N^{m+2}}\sum_{|J|\leq N}(|1+\frac{u_{0}}{Nz_{0}}|^{j_{0}-1}\dots|1+\frac{u_{m}}{Nz_{m}}|^{j_{m}})c_{J}z_{0}^{j_{0}-1}\dots z_{m}^{j_{m}}\bar{z}^{J}|\\
&\leq e^{\sum_{i=0}^{m}|\frac{u_{i}}{z_{i}}|}\frac{1}{N^{m+2}}\sum_{|J|\leq N}c_{J}j_{0}z_{0}^{j_{0}-1}\dots z_{m}^{j_{m}}\bar{z}^{J}\\
&=e^{\sum_{i=0}^{m}|\frac{u_{i}}{z_{i}}|}\frac{1}{N^{m+2}}\frac{\partial}{\partial z_{0}}S_{N}(z,z).\\
\end{split}
\end{equation}
By using \eqref{f} we see that $\frac{1}{N^{m+2}}\frac{\partial}{\partial z_{0}}S_{N}(z,z)$ converges. So there is a positive constant $M_{0}$
 such that $|\frac{1}{N^{m+2}}\frac{\partial}{\partial z_{0}}S_{N}(z,z)|\leq M_{0}$. In other words
\begin{equation}
|\frac{\partial}{\partial{u_{0}}}G_{N}(u)|\leq M_{0}e^{\sum_{i=0}^{m}|\frac{u_{i}}{z_{i}}|}.
\end{equation}
\end{proof}
Now by using lemma \eqref{n1} we see that $\{G_{N}\}$ is an equicontinuous sequence of holomorphic functions on $\bar{B}(0,1)\subset\mathbb{C}^{m+1}$ that is also uniformly bounded on $\bar{B}(0,1)$. So by using Arzel Ascoli Theorem, there is a subsequence like $\{G_{N_{j}}\}$ which converges uniformly on $\bar{B}(0,1)$. In the next Theorem we compute the limit of this subsequence and after that we prove that the whole sequence converges to the same limit.
\begin{thm}
If $z=(z_{0},\dots,z_{m})\in X\cap(\mathbb{C}^{*})^{m+1}$ and $u=(u_{0},\dots,u_{m})$ then $$\lim_{N\rightarrow\infty}\frac{1}{N^{m+1}}S_{N}(z+\frac{u}{N},z)=C_{\Omega, z, \mu,m}F_{m}(\beta(u))$$
where $C_{\Omega, z, \mu,m}$ is a constant that depends on $\Omega$, $z$, $\mu$, $m$ and
\begin{equation}
 F_{m}(t)=\int_{0}^{1}e^{ty}y^{m}dy \;\text{,}\; \beta(w)=\frac{d'\rho(z)\cdot w}{d'\rho(z)\cdot z}\:\text{for}\; w\in\mathbb{C}^{m+1}.
\label{11}
\end{equation}
\label{o}
\end{thm}
\begin{proof}
We already proved that there is a convergent subsequence of $G_{N}$, $G_{N_{j}}$, that converges uniformly o $\bar{B}(0,1)\subset\mathbb{C}^{m+1}$. Now by writing Taylor series for any $\{G_{N_{j}}\}$ around the origin we will have,
\begin{equation}
G_{N_{j}}(u)=\sum_{\alpha}\frac{\partial^{\alpha}}{\partial u^{\alpha}}G_{N_{j}}(0)\frac{u^{\alpha}}{\alpha !}.
\end{equation}
On the other hand if we let,
\begin{equation}
G(u)=lim_{j\rightarrow\infty}G_{N_{j}}(u),
\end{equation}
then
\begin{equation}
\frac{\partial^{\alpha}}{\partial u^{\alpha}}G(0)=\lim_{j\rightarrow\infty}\frac{\partial^{\alpha}}{\partial u^{\alpha}}G_{N_{j}}(0).
\end{equation}
Because each $G_{N_{j}}$ is holomorphic on $\mathbb{C}^{m+1}$ and they converge uniformly on $\bar{B}(0,1)$ to $G(u)$, so
\begin{equation}
\begin{split}
G(u)=\sum_{\alpha}\frac{\partial^{\alpha}}{\partial u^{\alpha}}G(0)\frac{u^{\alpha}}{\alpha !}
&=\sum_{\alpha}\lim_{j\rightarrow\infty}\frac{\partial^{\alpha}}{\partial u^{\alpha}}G_{N_{j}}(0)\frac{u^{\alpha}}{\alpha !}\\
&=\sum_{\alpha}\lim_{N_{j}\rightarrow\infty}\frac{1}{N_{j}^{m+|\alpha|+1}}\frac{\partial^{\alpha}}{\partial\bar{z}^\alpha}S_{N_{j}}(z,z)\frac{(u)^{\alpha}}{\alpha!}\\
&=s_{0}(z,z)t_{0}^{m+1}\sum_{\alpha}\int_{0}^{1}y^{m}\frac{(t_{0}y\frac{\partial\rho}{\partial z_{0}}u_{0}..\frac{\partial\rho}{\partial z_{m}}u_{m})^{\alpha}}{\alpha!}dy\\
&=s_{0}(z,z)t_{0}^{m+1}\int_{0}^{1}e^{yt_{0}(d'\rho(z)\cdot u)}y^{m}dy\\
&=s_{0}(z,z)t_{0}^{m+1}F_{m}(\frac{d'\rho(z)\cdot u}{d'\rho(z)\cdot z})\\
&=C_{\Omega, z, \mu,m}F_{m}(\beta(u)).\\
\end{split}
\end{equation}
Hence any convergent subsequence of
 \begin{equation}
 \{\frac{1}{N^{m+1}}S_{N}(z+\frac{u}{N},z)\},
 \end{equation}
 converges to $C_{\Omega, z, \mu,m}F_{m}(\beta(u)),$ and also we showed it is bounded. So it means
\begin{equation}
\lim_{N\rightarrow\infty}\frac{1}{N^{m+1}}S_{N}(z+\frac{u}{N},z)=C_{\Omega, z, \mu,m}F_{m}(\beta(u)).
\end{equation}
\end{proof}
\begin{thm}
If $z=(z_{0},\dots,z_{m})\in X\cap(\mathbb{C}^{*})^{m+1}$ and $u=(u_{0},\dots,u_{m})$, $v=(v_{0},\dots,v_{m})\in \mathbb{C}^{m+1}$ then
\begin{equation}
\lim_{N\rightarrow\infty}\frac{1}{N^{m+1}}S_{N}(z+\frac{u}{N},z+\frac{v}{N})=C_{\Omega, z,\mu,m}F_{m}(\beta(u)+\bar{\beta}(v)),
\end{equation}
\label{intro10}
\end{thm}
\subsection{Derivatives of partial szeg\"{o} kernel}
Our main tool for computing scaling limit correlation function is the Kac-Rice formula which for that we need to know derivatives of partial szeg\"{o} kernel.
 In this section we put our aim to compute scaling limit of derivative of partial szeg\"{o} kernel.
\begin{thm}
If $z=(z_{0},\dots,z_{m})\in X\cap(\mathbb{C}^{*})^{m+1}$ and $u=(u_{0},\dots,u_{m})$, $v=(v_{0},\dots,v_{m})\in \mathbb{C}^{m+1}$ then
\begin{equation}
\lim_{N\rightarrow\infty}\frac{1}{N^{m+2}}\frac{\partial}{\partial z_{i}}S_{N}(z+\frac{u}{N},z+\frac{v}{N})=s_{0}(z,z)t_{0}^{m+2}\frac{\partial\rho}{\partial z_{i}}F_{m+1}(\beta(u)+\bar{\beta}(v)),
\end{equation}
\begin{equation}
\lim_{N\rightarrow\infty}\frac{1}{N^{m+2}}\frac{\partial}{\partial \bar{z}_{i}}S_{N}(z+\frac{u}{N},z+\frac{v}{N})=s_{0}(z,z)t_{0}^{m+2}\frac{\partial\rho}{\partial\bar{z}_{i}}F_{m+1}(\beta(u)+\bar{\beta}(v)).
\end{equation}
\label{q}
\end{thm}
\begin{proof}
Let
\begin{equation}
G_{N}(u,v)=\frac{1}{N^{m+1}}S_{N}(z+\frac{u}{N},z+\frac{v}{N}),
\end{equation}
then
\begin{equation}
\begin{split}
\frac{\partial}{\partial u_{i}}G_{N}(u,v)
&=\frac{1}{N^{m+1}}\frac{\partial}{\partial z_{i}}S_{N}(z+\frac{u}{N},z+\frac{v}{N})\frac{1}{N}\\
&=\frac{1}{N^{m+2}}\frac{\partial}{\partial z_{i}}S_{N}(z+\frac{u}{N},z+\frac{v}{N}).\\
\end{split}
\end{equation}
On the other hand
\begin{equation}
\begin{split}
\lim_{N\rightarrow\infty}\frac{1}{N^{m+2}}\frac{\partial}{\partial z_{i}}S_{N}(z+\frac{u}{N},z+\frac{v}{N})
&=\lim_{N\rightarrow\infty}\frac{\partial}{\partial u_{i}}G_{N}(u,v)\\
&=\frac{\partial}{\partial u_{i}}\lim_{N\rightarrow\infty}G_{N}(u,v)\\
&=\frac{\partial}{\partial u_{i}}(s_{0}(z,z)t_{0}^{m+1}F_{m}(\beta(u)+\bar{\beta}(v)))\\
&=s_{0}(z,z)t_{0}^{m+1}\frac{\partial}{\partial u_{i}}F_{m}((\beta(u)+\bar{\beta}(v))\\
&=s_{0}(z,z)t_{0}^{m+1}\frac{\partial\beta(u)}{\partial u_{i}}F'_{m}(\beta(u)+\bar{\beta}(v))\\
&=s_{0}(z,z)t_{0}^{m+2}\frac{\partial\rho}{\partial z_{i}}F'_{m}(\beta(u)+\bar{\beta}(v))\\
&=s_{0}(z,z)t_{0}^{m+2}\frac{\partial\rho}{\partial z_{i}}F_{m+1}(\beta(u)+\bar{\beta}(v)),\\
\end{split}
\end{equation}
 and similarly by following the same proof we can show that
 \begin{equation}
 \lim_{N\rightarrow\infty}\frac{1}{N^{m+2}}\frac{\partial}{\partial \bar{z}_{i}}S_{N}(z+\frac{u}{N},z+\frac{v}{N})=s_{0}(z,z)t_{0}^{m+2}\frac{\partial\rho}
 {\partial\bar{z}_{i}}F_{m+1}(\beta(u)+\bar{\beta}(v)).
 \end{equation}
 \end{proof}
\begin{thm}
If $z=(z_{0},\dots,z_{m})\in X\cap(\mathbb{C}^{*})^{m+1}$ and $u=(u_{0},\dots,u_{m})$, $v=(v_{0},\dots,v_{m})\in \mathbb{C}^{m+1}$ then
\begin{equation}
\begin{split}
\lim_{N\rightarrow\infty}\frac{1}{N^{m+3}}\frac{\partial^{2}}{\partial \bar{z}_{i}\partial z_{j}}S_{N}(z+\frac{u}{N},z+\frac{v}{N})=
s_{0}(z,z)t_{0}^{m+3}\frac{\partial\rho}{\partial\bar{z}_{i}}\frac{\partial\rho}{\partial z_{j}}F_{m+2}(\beta(u)+\bar{\beta}(v)).
\end{split}
\end{equation}
\label{r}
\end{thm}
\begin{proof}
\begin{equation}
\begin{split}
\lim_{N\rightarrow\infty}\frac{1}{N^{m+3}}\frac{\partial^{2}}{\partial \bar{z}_{i}\partial z_{j}}S_{N}(z+\frac{u}{N},z+\frac{v}{N})
&=\lim_{N\rightarrow\infty}\frac{1}{N^{m+1}}\frac{\partial^{2}}{\partial v_{i}\partial u_{j}}G_{N}(u,v)\\
&=\frac{\partial^{2}}{\partial v_{i}\partial u_{j}}\lim_{N\rightarrow\infty}\frac{1}{N^{m+1}}G_{N}(u,v)\\
&=\frac{\partial^{2}}{\partial v_{i}\partial u_{j}}(s_{0}(z,z)t_{0}^{m+1}F_{m}(\beta(u)+\bar{\beta}(v)))\\
&=s_{0}(z,z)t_{0}^{m+1}\frac{\partial^{2}}{\partial v_{i}\partial u_{j}}F_{m}(\beta(u)+\bar{\beta}(v))\\
&=s_{0}(z,z)t_{0}^{m+3}\frac{\partial\rho}{\partial\bar{z}_{i}}\frac{\partial\rho}{\partial z_{j}}F_{m}''(\beta(u)+\bar{\beta}(v))\\
&=s_{0}(z,z)t_{0}^{m+3}\frac{\partial\rho}{\partial\bar{z}_{i}}\frac{\partial\rho}{\partial z_{j}}F_{m+2}(\beta(u)+\bar{\beta}(v)).\\
\end{split}
\end{equation}
\end{proof}
\section{Scaling limit Distributions}
\label{Scaling limit Distributions}
We now have all the ingredients that we need to compute the Scaling limit distribution functions. We expect the scaling limits to exist and depend only on the $m,z, X$. Bleher, Shiffman, and Zelditch in \cite{BSZ1} gave a formula for the $l$- point zero correlation function in terms of the projection kernel and its first and second derivatives. For the 1-point correlation function we define the matrices
\begin{equation}
\bigtriangleup_{N}=\begin{pmatrix} A_{N}&B_{N}\\ B_{N}^{*}&C_{N}\end{pmatrix},\:\text{where}:
\end{equation}
\begin{equation}
A_{N}=S_{N}(z+\frac{u}{N},z+\frac{u}{N}),
\end{equation}
\begin{equation}
B_{N}=(\frac{\partial}{\partial\bar{z}_{i}} S_{N}(z+\frac{u}{N},z+\frac{u}{N}))_{0\leq i\leq m},
\end{equation}
\begin{equation}
C_{N}=(\frac{\partial^{2}S_{N}}{\partial z_{i}\partial \bar{z}_{j}}(z+\frac{u}{N},z+\frac{u}{N}))_{0\leq i,j\leq m},
\end{equation}
\begin{equation}
\Lambda_{N}=C_{N}-(B_{N})^{*}A_{N}^{-1}B_{N}.\\
\end{equation}
Writing
\begin{equation}
E_{\mu,X}^{N}([Z_{f}]\wedge\frac{\omega^{m}}{m!})=D^{N}_{\mu,X}(z)\frac{\omega^{m+1}}{(m+1)!},
\end{equation}
by using the general formula given in \cite{BSZ1} for the l-point density functions we get
\begin{equation}
D_{\mu,X}^{N}=\frac{1}{\pi}\frac{\sum_{i=0}^{m}(\Lambda_{N})_{i,i}}{det(A_{N})}.
\end{equation}
Our goal is to compute,
\begin{equation}
\lim_{N\to\infty}\frac{D_{\mu,X}^{N}(z+\frac{u}{N})}{N^{2}}=
\lim_{N\to\infty}\frac{1}{\pi}\frac{\sum_{i=0}^{m}\frac{(\Lambda)_{i,i}}{N^{m+3}}}{\frac{det(A_{N})}{N^{m+1}}}.
\end{equation}
We define
\begin{equation}
P=(\frac{\partial\rho}{\partial\bar{z}_{0}},\dots,\frac{\partial\rho}{\partial\bar{z}_{m}}).
\label{alphap}
\end{equation}
So by using the definition of $P$ we can simplify each formula that we computed for the scaling limit of szeg\"{o} kernel and its derivatives.
 Now if we use theorems \eqref{intro10}, \eqref{q}, \eqref{r} then we will have:
 \begin{equation}
\lim_{N\to\infty}\frac{A_{N}}{N^{m+1}}=s_{0}(z,z)t_{0}^{m+1}F_{m}(\beta(u)+\bar{\beta}(u)),
\label{s}
\end{equation}
\begin{equation}
\begin{split}
\lim_{N\to\infty}\frac{B_{N}}{N^{m+2}}
&=s_{0}(z,z)t_{0}^{m+2}(\frac{\partial\rho(z)}{\partial\bar{z}_{0}},\dots,\frac{\partial\rho(z)}{\partial\bar{z}_{m}})F_{m+1}(\beta(u)+\bar{\beta}(u))\\
&=s_{0}(z,z)t_{0}^{m+2}F_{m+1}(\beta(u)+\bar{\beta}(v))P, \\
\end{split}
\label{t}
\end{equation}
\begin{equation}
\begin{split}
\lim_{N\to\infty}\frac{C_{N}}{N^{m+3}}
&=s_{0}(z,z)t_{0}^{m+3}(\frac{\partial\rho(z)}{\partial z_{i}}\frac{\partial\rho(z)}{\partial\bar{z}_{j}}F_{m+2}(\beta(u)+\bar{\beta}(u)))_{0\leq i,j\leq m}\\
&=s_{0}(z,z)t_{0}^{m+3}F_{m+2}(\beta(u)+\bar{\beta}(u))P^{*}P.\\
\end{split}
\label{u}
\end{equation}
Now if we plug results that we have from equations \eqref{s}, \eqref{t}, \eqref{u} in,
\begin{equation}
\lim_{N\to\infty}\frac{\Lambda_{N}}{N^{m+3}}=\lim_{N\to\infty}(\frac{C_{N}}{N^{m+3}}-(\frac{B_{N}}{N^{m+2}})^{*}(\frac{A_{N}}{N^{m+1}})^{-1}
(\frac{B_{N}}{N^{m+2}})),
\end{equation}
then we will have,
\begin{equation}
\lim_{N\to\infty}\frac{\Lambda_{N}}{N^{m+3}}=s_{0}(z,z)t_{0}^{m+3}(F_{m+2}(\beta(u)+\bar{\beta}(u))-
\frac{F_{m+1}^{2}(\beta(u)+\bar{\beta}(u))}{F_{m}(\beta(u)+\bar{\beta}(u))})P^{*}P.
\label{v}
\end{equation}
Consequently
\begin{equation}
\lim_{N\to\infty}\frac{(\Lambda_{N})_{i,i}}{N^{m+3}}
=s_{0}(z,z)t_{0}^{m+3}(F_{m+2}(\beta(u)+\bar{\beta}(u))-\frac{F_{m+1}^{2}(\beta(u)+\bar{\beta}(u))}{F_{m}(\beta(u)+\bar{\beta}(v))})(P^{*}P)_{i,i}.
\label{w}
\end{equation}
We know that $||P||^{2}=\sum_{i=0}^{m}(P^{*}P)_{i,i}$, so we have
\begin{equation}
\lim_{N\to\infty}\sum_{i=0}^{m}\frac{(\Lambda_{N})_{i,i}}{N^{m+3}}
=s_{0}(z,z)t_{0}^{m+3}(F_{m+2}(\beta(u)+\bar{\beta}(u))-\frac{F_{m+1}^{2}(\beta(u)+\bar{\beta}(u))}{F_{m}(\beta(u)+\bar{\beta}(u))})||P||^{2}.
\end{equation}
\label{x}
\begin{thm}
Let $D_{\mu,X}^{N}$ be the expected zero density for the ensemble $(\mathcal{P}_{N},\gamma_{N})$ then

  $$\lim_{N\rightarrow\infty}\frac{1}{N^{2}}D_{\mu,X}^{N}(z+\frac{u}{N})=D_{z,X}^{\infty}(u),$$
  where
 $$D_{z,X}^{\infty}(u)=\frac{(\beta(P))^{2}}{||P||^{2}\pi}(\log F_{m})^{''}(\beta(u)+\bar{\beta}(u)),$$
where $P$ is defined at \eqref{alphap}.
\label{y}
\end{thm}
\begin{proof}
\begin{equation}
\begin{split}
D^{\infty}_{z,X}(u)&=\frac{1}{\pi}\lim_{N\rightarrow\infty}\frac{D_{\mu,X}^{N}(z+\frac{u}{N})}{N^{2}}
=\lim_{N\to\infty}\frac{1}{\pi}\frac{\sum_{i=0}^{m}\frac{(\Lambda)_{i,i}}{N^{m+3}}}{\frac{det(A_{N})}{N^{m+1}}}\\
&=\frac{s_{0}(z,z)t_{0}^{m+3}(F_{m+2}(\beta(u)+\bar{\beta}(u))-\frac{F_{m+1}^{2}(\beta(u)+\bar{\beta}(u))}{F_{m}(\beta(u)+\bar{\beta}(u))})}{s_{0}(z,z)t_{0}^{m+1}F_{m}(\beta(u)+\bar{\beta}(u))}||P||^{2}\\
&=\frac{1}{\pi}t_{0}^{2}\frac{F_{m+2}(\beta(u)+\bar{\beta}(u))F_{m}(\beta(u)+\bar{\beta}(u))-F_{m+1}^{2}(\beta(u)+\bar{\beta}(u))}{F_{m}(\beta(u)+\bar{\beta}(u))^{2}}||P||^{2}\\
&=\frac{1}{\pi}(t_{0}||P||)^{2}(\log F_{m})^{''}(\beta(u)+\bar{\beta}(u))\\
&=\frac{(\beta(P))^{2}}{||P||^{2}\pi}(\log F_{m})^{''}(\beta(u)+\bar{\beta}(u)).\\
\end{split}
\end{equation}
\end{proof}
\section{The scaling limit zero correlation function}
\label{The scaling limit zero correlation function}
Let $z\in X\cap(\mathbb{C}^{*})^{m+1}$ and $u\in \mathbb{C}^{m+1}$. So the scaling covariant matrix $\bigtriangleup_{N}(u)$ is
\begin{equation}
\bigtriangleup_{N}(u)=\begin{pmatrix} A_{N}(u)&B_{N}(u)\\ B_{N}^{*}(u)&C_{N}(u)\end{pmatrix},
\end{equation}
where
\begin{equation}
A_{N}(u) =\begin{pmatrix} S_{N}(z+\frac{u}{N},z+\frac{u}{N})&S_{N}(z+\frac{u}{N},z)\\ S_{N}(z,z+\frac{u}{N})&S_{N}(z,z)\end{pmatrix},
\end{equation}
\begin{equation}
B_{N}(u) =\begin{pmatrix} B_{N}^{1}(u)&B_{N}^{2}(u)\\ B_{N}^{3}(u)&B_{N}^{4}(u)\end{pmatrix},
\end{equation}
such that
\begin{equation}
\begin{split}
&B_{N}^{1}(u)=(\frac{\partial}{\partial\bar{z}_{i}}S_{N}(z+\frac{u}{N},z+\frac{u}{N}))_{0\leq i\leq m},\\
&B_{N}^{2}(u) =(\frac{\partial}{\partial\bar{z}_{i}}S_{N}(z+\frac{u}{N},z))_{0\leq i\leq m},\\
&B_{N}^{3}(u) =(\frac{\partial}{\partial\bar{z}_{i}}S_{N}(z,z+\frac{u}{N}))_{0\leq i\leq m},\\
&B_{N}^{4}(u)=(\frac{\partial}{\partial\bar{z}_{i}}S_{N}(z,z))_{0\leq i\leq m},\\
\end{split}
\end{equation}
\begin{equation}
C_{N}(u)=\begin{pmatrix} C_{N}^{1,1}(u)&C_{N}^{1,2}(u)\\ C_{N}^{2,1}(u)&C_{N}^{2,2}(u)\end{pmatrix},
\end{equation}
where
\begin{equation}
\begin{split}
&C_{N}^{1,1}(u)=(\frac{\partial^{2}S_{N}}{\partial z_{i}\partial\bar{z}_{j}}(z+\frac{u}{N},z+\frac{u}{N}))_{0\leq i,j\leq m},\\
&C_{N}^{1,2}(u)=(\frac{\partial^{2}S_{N}}{\partial z_{i}\partial\bar{z}_{j}}(z+\frac{u}{N},z))_{0\leq i,j\leq m},\\
&C_{N}^{2,1}(u)=(\frac{\partial^{2}S_{N}}{\partial z_{i}\partial\bar{z}_{j}}(z,z+\frac{u}{N}))_{0\leq i,j\leq m},\\
&C_{N}^{2,2}(u)=(\frac{\partial^{2}S_{N}}{\partial z_{i}\partial\bar{z}_{j}}(z,z))_{0\leq i,j\leq m}.\\
\end{split}
\end{equation}
So the scaling limits of the matrices, $A_{N}$, $B_{N}$, $C_{N}$ are
\begin{equation}
\begin{split}
 A_{\infty}(u)&=\lim_{N\to\infty}\frac{1}{N^{m+1}}A_{N}\\&=s_{0}(z,z)t_{0}^{m+1}\begin{pmatrix} F_{m}(\beta(u)+\bar{\beta}(u))&F_{m}(\beta(u))\\ F_{m}(\bar{\beta}(u))&F_{m}(0)\end{pmatrix},\\
 \end{split}
\end{equation}
\begin{equation}
\begin{split}
B_{\infty}(u)&=\lim_{N\to\infty}\frac{1}{N^{m+2}}B_{N}(u)\\&=s_{0}(z,z)t_{0}^{m+2}\begin{pmatrix} F_{m+1}(\beta(u)+\bar{\beta}(u))\bar{P}&F_{m+1}(\beta(u))\bar{P}\\  F_{m+1}(\bar{\beta}(u))\bar{P}&F_{m+1}(0)\bar{P}\end{pmatrix},\\
\end{split}
\end{equation}
\begin{equation}
\begin{split}
C_{\infty}(u)&=\lim_{N\to\infty}\frac{1}{N^{m+3}}C_{N}(u)\\ &=s_{0}(z,z)t_{0}^{m+3}\begin{pmatrix} F_{m+2}(\beta(u)+\bar{\beta}(u))P^{*}P&F_{m+2}(\beta(u))P^{*}P\\
  F_{m+2}(\bar{\beta}(u))P^{*}P&F_{m+2}(0)P^{*}P\end{pmatrix}.\\
\end{split}
\end{equation}
To simplify the computations, we define the two by two matrix,
\begin{equation}
G_{m}(x)=\begin{pmatrix} F_{m}(x+\bar{x})&F_{m}(x)\\ F_{m}(\bar{x})&F_{m}(0)\end{pmatrix},
\label{z}
\end{equation}
where all the entries of the matrix $G_{m}$ are identified by $F_{m}$. If $x\in\mathbb{C}^{*}$ then
\begin{equation}
F_{m}(x)F_{m}(\bar{x})< F_{m}(0)F_{m}(x+\bar{x})=\frac{1}{m}F_{m}(x+\bar{x}).
\label{holder}
\end{equation}
So for nonzero $x\in\mathbb{C}$ the matrix $G_{m}(x)$ is invertible, therefore
\begin{equation}
 Q_{m}(x)=G_{m+2}(x)-G_{m+1}(x)G_{m}(x)^{-1}G_{m+1}(x),
\label{w}
\end{equation}
 is a well defined two by two matrix on $\mathbb{C}^{*}$.
 This means that $G_{m}(\beta(u))^{-1}$ is a well-defined matrix. Hence we have
\begin{equation}
\begin{split}
 \Lambda_{\infty}&=C_{\infty}- B_{\infty}^{*}A_{\infty}^{-1}B_{\infty}\\
 &=s_{0}(z,z)t_{0}^{m+3}\begin{pmatrix}Q_{1,1}P^{*}P&Q_{1,2}P^{*}P\\ Q_{2,1}P^{*}P&Q_{2,2}P^{*}P\end{pmatrix}.\\
\end{split}
\end{equation}
 Our goal is to compute scaling limit normalized pair correlation function,
\begin{equation}
 \tilde{K}_{z,X}^{\infty}(u)=\lim_{N\to\infty}\frac{K_{\mu,X}^{N}(z+\frac{u}{N},z)}{D_{\mu,X}^{N}(z+\frac{u}{N})D_{\mu,X}^{N}(z)},
\label{y1}
\end{equation}
  where
\begin{equation}
E^{N}_{\mu,X}([Z_{f}(z)]\wedge[Z_{f}(w)]\wedge\frac{\omega_{z}^{m}}{(m)!}\wedge\frac{\omega_{w}^{m}}{(m)!})=
K^{N}_{\mu,X}(z,w)\frac{\omega_{z}^{m+1}}{(m+1)!}\wedge\frac{\omega_{w}^{m+1}}{(m+1)!},
\label{w1}
\end{equation}
\begin{thm}
Let $\tilde{K}_{\mu,X}^{N}(z,w)$ be the normalized pair correlation function for the probability space $(P_{N},\gamma_{N})$ and choose $u\in\mathbb{C}^{m+1}$ such that $u\notin T_{z}^{h}X$. Then,
$$\lim_{N\rightarrow\infty}\frac{1}{N^{4}}K_{\mu,X}^{N}(z+\frac{u}{N},z)=K^{\infty}_{z,X}(u),$$
$$\lim_{N\rightarrow\infty}\tilde{K}^{N}_{\mu,X}(z+\frac{u}{N},z)=\tilde{K}^{\infty}_{z,X}(u),$$
where
$$K^{\infty}_{z,X}(u)=\frac{1}{\pi^{2}||P||^{4}}\frac{perm(Q_{m}(\beta(u)))}{det(G_{m}(\beta(u)))}(\beta(P))^{4},$$
$$\tilde{K}^{\infty}_{z,X}(u)=\frac{1}{(\log F_{m})^{''}(\beta(u)+\bar{\beta}(u))(\log F_{m})^{''}(0)}\frac{perm(Q_{m}(\beta(u)))}{det(G_{m}(\beta(u)))},$$
where $K^{N}_{\mu,X}(z,w)$, $\tilde{K}^{N}_{\mu,X}(z,w)$ are defined in \eqref{intro16}, \eqref{intro15}.
\label{z}
\end{thm}
\begin{proof}
At first by using Kac-Rice formula we compute $\frac{1}{N^{4}}K_{\mu,X}^{N}(z+\frac{u}{N},z)$ and then by using Theorem \eqref{y} we compute the scaling limit for the normalized pair correlation function.
\begin{equation}
\begin{split}
\frac{1}{N^{4}}K_{\mu,X}^{N}(z+\frac{u}{N},z)
&= \frac{(\sum_{i=0}^{m}\frac{\Lambda_{i,i}^{N}}{N^{m+3}})(\sum_{i=m+1}^{2m}\frac{\Lambda_{i,i}^{N}}{N^{m+3}})} {\pi^{2}\frac{det(A^{N})}{N^{2m+2}}}+\\
&\frac{\sum_{i=m+1}^{2m}
\frac{\Lambda_{1,i}^{N}}{N^{m+3}}\frac{\Lambda_{i,1}^{N}}{N^{m+3}}+\dots+\sum_{i=m+1}^{2m}\frac{\Lambda_{m,i}^{N}}{N^{m+3}}\frac{\Lambda_{i,m}^{N}}{N^{m+3}}}{\pi^{2}\frac{det(A^{N})}{N^{2m+2}}}\\ &\rightarrow\frac{(Q_{1,1}Q_{2,2}+Q_{1,2}Q_{2,1})||P||^{4}t_{0}^{4}}{\pi^{2}det(G_{m}(\beta(u)))}\\
&=\frac{1}{\pi^{2}||P||^{4}}\frac{perm(Q_{m}(\beta(u)))}{det(G_{m}(\beta(u)))}(\beta(P))^{4}.\\
\end{split}
\label{z1}
\end{equation}
 Now we are ready to give a general formula for $\tilde{K}_{z,X}^{\infty}(u)$. If we use equation \eqref{z1} then,
 \begin{equation}
 \begin{split}
 \tilde{K}_{z,X}^{\infty}(u)&=\lim_{N\to\infty}\frac{K_{\mu,X}^{N}(z+\frac{u}{N},z)}{D_{\mu,X}^{N}(z+\frac{u}{N})D_{\mu,X}^{N}(z)}\\
 &=\lim_{N\to\infty}\frac{\frac{K_{\mu,X}^{N}(z+\frac{u}{N},z)}{N^{4}}}{\frac{D_{\mu,X}^{N}(z+\frac{u}{N})}{N^{2}}\frac{D_{\mu,X}^{N}(z)}{N^{2}}}\\
 &=\frac{\frac{perm(Q_{m}(\beta(u)))(||P||t_{0})^{4}}{\pi^{2}det(G_{m}(\beta(u)))}}{(\frac{||P||^{2}t_{0}^{2}}{\pi}F_{m}(\beta(u)+\bar{\beta}(u)))(\frac{||P||^{2}t_{0}^{2}}{\pi}F_{m}(0))}\\
 &=\frac{1}{(\log F_{m})^{''}(\beta(u)+\bar{\beta}(u))(\log F_{m})^{''}(0)}\frac{perm(Q_{m}(\beta(u)))}{det(G_{m}(\beta(u)))}.\\
 \end{split}
 \end{equation}
\end{proof}


\begin{thebibliography}{11}
\bibitem[1]{BFG} Beals, C. Fefferman and R. Grossman, \emph{Strictly pseudoconvex domains in $\mathbb{C}^{n}$}, Bull. Amer. Math. Soc. 8 (1983), 125–322.

\bibitem[2]{BS1} T. Bloom and B. Shiffman, \emph{Zeros Of Random Polynomials On $\mathbb{C}^{m}$}, Math. Res. Lett. 14(2007), no. 3, 469-479.

\bibitem[3]{BS2} L. Boutet de Monvel and J. Sj¨ostrand, \emph{Sur la singularit´e des noyaux de Bergman et de Szeg¨o}, Asterisque 34–35 (1976), 123–164

\bibitem[4]{BSZ1} P. Bleher and B. Shiffman, and S. Zelditch, \emph{Universality and scaling of correlations between zeros on complex manifolds}, Invent. Math. 142(2000), no. 2, 351-395.
\bibitem[5]{BSZ2} P. Bleher and B. Shiffman, and S. Zelditch, \emph{Correlations between zeros and supersymmetry}, Commun. Math. Phys. 224 (2001), 255-269.
\bibitem[6]{Ha} J. H. Hammersley, \emph{The zeros of a random polynomial}, Proceeding of the Third Berkeley Symposium on Mathematical Statistics and Probability, 1954-1955, vol. 2, University of California Press, California, 1956, pp. 89-111.
\bibitem[7]{Ho} L. Hormander, \emph{The Analysis of Linear Partial Differential Operators I}, Grund. Math. Wiss. 256,
Springer-Verlag, N.Y. (1983).
\bibitem[8]{Kac1} M. Kac, \emph{On the average number of real roots of a random algebraic equation}, Bull. Amer. Math.
Soc. 49 (1943), 314–320.
\bibitem[9]{Kac2} M. Kac, \emph{On the average number of real roots of a random algebraic equation.II}, Proc. London Math. Soc. (2) 50 (1949), 390-408.s
\bibitem[10]{SK} S.Krantz,\emph{Function theory of several complex variables} (2ed., AMS, 1992)(ISBN 0534170889)
\bibitem[11]{SZ} B. Shiffman, and S. Zelditch, \emph{Equilibrium distribution of zeros of random polynomials},  Int. Math. Res. Not. 2003, 25-49.
\bibitem[12]{Ze} S. Zelditch, \emph{Szeg\"{o} kernels and a theorem of Tian}, Int. Math. Res. Notices 6 (1998), 317–331.
\end{thebibliography}
\end{document}